\documentclass[12pt,a4paper,reqno]{amsart} 
\usepackage{amssymb}
\usepackage{latexsym}
\usepackage{amsmath}
\usepackage{mathrsfs}

\def\R{{\mathbb{R}}}
\def\N{{\mathbb{N}}}

\def\pd{\mathrm{pd}}
\def\hyp{\mathrm{hyp}}
\def\reg{\mathrm{reg}}
\def\osc{\mathrm{osc}}
\usepackage{color}
\usepackage[dvipsnames]{xcolor}

\newcommand{\<}{\langle}
\renewcommand{\>}{\rangle}
\renewcommand{\phi}{\varphi}

\newcommand{\Op}{\operatorname{Op}}
\newcommand{\SG}{\operatorname{SG}}

\newcommand{\norm}[1]{\langle#1\rangle}

\def\eps{{\varepsilon}}

\def\scS{{\mathscr S}}
\def\scSp{{\mathscr S}^\prime}

\def\jap#1{\langle {#1} \rangle}
\def\norm#1{\langle #1 \rangle}

\def\SGs{S}

\def\SGH#1#2#3#4#5{S^{#1,#2}\{#3,#4\}_{#5}^{\textnormal{hyp}}}
\def\SGO#1#2#3#4#5{S^{#1,#2}\{#3,#4\}_{#5}^{\textnormal{reg}}}
\def\SGP#1{S_{#1}^{\textnormal{pd}}}

\newtheorem{thm}{Theorem}[section]
\newtheorem{prop}[thm]{Proposition}
\newtheorem{lem}[thm]{Lemma}
\newtheorem{cor}[thm]{Corollary}

\theoremstyle{definition}
\newtheorem{exmp}[thm]{Example}
\newtheorem{defn}[thm]{Definition}

\theoremstyle{remark}
\newtheorem{rem}[thm]{Remark}
\renewcommand{\theta}{\vartheta}

\newcommand{\beq}{\begin{eqnarray}}
\newcommand{\eeq}{\end{eqnarray}}
\newcommand{\beqst}{\begin{eqnarray*}}
\newcommand{\eeqst}{\end{eqnarray*}}
\newcommand{\be}{\begin{equation}}
\newcommand{\ee}{\end{equation}}

\newcommand{\iy}{\infty}

\newcommand{\pa}{\partial}

\newcommand{\al}{\alpha}

\newcommand{\va}{\varphi}
\newcommand{\il}{\int\limits}

\newcommand{\ve}{\varepsilon}
\newcommand{\si}{\sigma}

\newcommand{\weight}{\bigg(\frac{\lambda(t)}{\Lambda(t)}\ln\bigg(\frac{1}{\Lambda(t)}\bigg)\bigg)^k}
\newcommand{\kR}{\mbox{\scriptsize\rm R\hspace{-0,7em}\rule{0,08ex}{1,55ex}
\hspace{0,4em}}}
\DeclareMathOperator*{\esssup}{ess\sup}

\addtocounter{section}{0}
\title[Hyperbolic Cauchy Problems on $\R^{d}$]{Global Wellposedness of a Class of \\ Weakly Hyperbolic Cauchy Problems\\ with Variable Multiplicities on $\R^d$}

\author{Sandro Coriasco}
\address{Dipartimento di Matematica ``G. Peano'', Universit\`a degli Studi di Torino, Torino, Italy}
\email{sandro.coriasco@unito.it}

\author{Giovanni Girardi}
\address{Dipartimento di Ingegneria Industriale e Scienze Matematiche, Universit\`a Politecnica delle Marche, Ancona, Italy}
\email{g.girardi@staff.univpm.it}

\author{N. Uday Kiran}
\address{%
Department of Mathematics and Computer Science\\
Sri Sathya Sai Institute of Higher Learning, Puttaparthi\\
Andhra Pradesh, India}
\email{nudaykiran@sssihl.edu.in}

\keywords{Weakly hyperbolic operators, Variable multiplicities, Cauchy problem, Polynomially bounded coefficients, Loss of derivatives, Loss of decay, Sobolev-Kato spaces}

\subjclass[2000]{35L15, 35S30}

\begin{document}

\begin{abstract}
We study a class of weakly hyperbolic Cauchy problems on $\R^d$, involving linear operators with characteristics of variable multiplicities, whose coefficients are unbounded in the space variable. The behaviour in the time variable is governed by a suitable ``shape function''. We develop a parameter-dependent symbolic calculus, 
corresponding to an appropriate subdivision of the phase space. By means of such calculus, a parametrix can be constructed, in terms of (generalized) Fourier integral operators 
naturally associated with the employed symbol class. Further, employing the parametrix, we prove $\scS(\R^{d})$-wellposedness and give results about the global decay and 
regularity of the solution, within a scale of weighted Sobolev space. 
\end{abstract}

\maketitle

\section{Introduction}\label{sec:intro}
%
%
%
In this paper we deal with the global decay and regularity properties of the solution of certain weakly hyperbolic Cauchy problems, associated
with linear partial differential operators with smooth coefficients,
polynomially growing in the space variable $x$. Setting
\[
	L = -D_{t}^{2}u+\sum_{j=1}^{d}\left(a_{j}(t,x)D_{j}^{2}+b_{j}(t,x)D_{j}\right)+c(t,x),
\]
our model Cauchy problem is 
\begin{equation}\label{eq:main}
\begin{cases}
Lu(t,x)&\hspace*{-10pt}=g(t,x),  \; (t,x) \in [0,T]\times\mathbb{R}^{d}, \\
\phantom{L}u(0,x) &\hspace*{-10pt}= \phi(x), u_{t}(0,x)=\psi(x), \; x\in\R^d,
\end{cases}
\end{equation}
where $D_{j}=-i\partial/\partial x_{j}$, $D_{t}=-i\partial/\partial t$  and $a_{j}(t,x), b_{j}(t,x),c(t,x)$, $j=1,\dots,d$, are functions in $ C([0,T]\times \mathbb{R}^d)\cap C^\infty((0,T]\times \mathbb{R}^d)$, 
for (a suitably small) $T>0$. We define
\begin{equation}
\label{eq:amplitude_a} 
a(t,x,\xi)=\sum^{d}_{j=1}\left(a_{j}(t,x)\xi^{2}_{j}+b_{j}(t,x)\xi_{j}\right)+c(t,x),
\end{equation}
and assume that $a(t,\cdot,D)=\Op(a(t))$ is a differential operator with coefficients that are smooth on $(0,T]\times\R^d$, 
but, possibly, with characteristic roots of variable multiplicity
and \textit{fast oscillations} (see \cite{Rei2}) at $t=0$. Compared with the hyperbolic operators with variable multiplicities studied in \cite{AAC},
we here consider different conditions on the characteristic roots. Namely, the behaviour with respect to the time variable $t$ is
governed by a shape function $\lambda(t)$, as explained in Section \ref{subs:lambda} below. The function $a$ is assumed to be a family of symbols belonging to one of the  
Weyl-H\"ormader classes $S(m,g)$, namely, where the metric on the phase space $\R^{2d}$ is
\begin{equation}\label{SG_metric}
	g_{(x,\xi )}(y,\eta ) = \norm{x}^{-2} |y|^2 + \norm{\xi}^{-2} |\eta |^2,
\end{equation}
where $\norm{\eta}$ stands for\footnote{The modification in the definition of the weight $\norm{y}$, usually given by $\sqrt{1+|y|^2}$, is due to the presence,
in the sequel, of logarithms of products of the type $\norm{x}\norm{\xi}$, which need not to vanish anywhere on $\R^{2d}$. All the usual symbol 
estimates and definitions are completely equivalent to those involving the standard definition of the weight.} 
$(e+|\eta|^2)^{1/2}, \eta\in \mathbb{R}^d$. Most often, we will employ weights of the form $m(x,\xi)=\norm{x}^{m}\norm{\xi} ^{\mu}$,
that is, the so-called $\SG$-symbols, in their most standard form $\SGs^{m,\mu}(\R^{2d})$, will be involved (cf. Cordes \cite{Cord}, Parenti \cite{Pa72}). 
However, we will also need more general symbols of $\SG$-type, that is, those associated with the metric 
\begin{equation*}
	g_{(x,\xi )}(y,\eta ) = \norm{x}^{-2r_1}\norm{\xi}^{2\rho _1} |y|^2 +
\norm{x}^{2r_2}\norm{\xi}^{-2\rho _2} |\eta |^2.
\end{equation*}
In this case, with the same choice of weight above, we will denote $S(m,g)=\SGs ^{m,\mu }
_{(r_j,\rho _j)}(\R^{2d})=\SGs^{m,\mu}_{r_1,r_2,\rho_1,\rho_2}(\R^{2d})$. If
$$
0\le r_2\le r_1\le 1\quad \text{and}\quad
0\le \rho _1\le \rho _2\le 1,
$$
then $g$ is feasible and $m$ is $g$-continuous in this case. If, in addition,
$r_2,\rho _1<1$, then $g$ is strongly feasible (see Coriasco and Toft \cite{CoTo}, H\"ormander \cite{Ho1}, Nicola and Rodino \cite{Nicola_Rodino}). Such conditions will actually be satisfied in our analysis below.

We will also assume that the symbols family $a(t)$ admits a lower bound, namely, 
\begin{equation}\label{ellipticity_3}
	a(t,x,\xi)\geq C\lambda(t)^{2}\norm{x}^2\norm{\xi}^2,
\end{equation} 
for a positive constant $C$. Notice that \eqref{ellipticity_3} does not imply ellipticity of $a(t)$ in the $\SG$-classes for
$t=0$, see Section \ref{subs:sgclasses} below for details. Then, the operator in \eqref{eq:main} will turn out to be weakly hyperbolic 
with variable multiplicities, since it will hold $\lambda(0)=0$, $\lambda(t)>0$, $t\in(0,T]$.

The main new feature, compared with similar approaches in different symbol classes, is allowing coefficients which are unbounded with respect to $x\in\R^d$.
Indeed, we will assume the following polynomial upper bounds on the symbols family $a$ and its derivatives: for all $\alpha, \beta\in\N^d$ and $k\in \N$, 
\begin{equation}\label{eq:sym_inequality}
|D_{t}^{k}D^{\alpha}_{x}D_{\xi}^{\beta}a(t,x,\xi)|\leq C_{k\alpha\beta}\,\lambda(t)^2\norm{x}^{2-|\alpha|}\norm{\xi}^{2-|\beta|} \Sigma(t)^{k},
\end{equation}
for some $C_{k\alpha\beta}>0$ and every $(t,x,\xi)\in (0,T]\times\R^{d}\times\R^d$.
In \eqref{eq:sym_inequality}, the non-negative function $\Sigma(t)$ captures the `local' weakening of `generalized' Lipschitz conditions, as described in Section \ref{subs:zones} below.
Inequality \eqref{eq:sym_inequality} means that, for all $t\in [0,T]$, the complete symbol $a(t,\cdot,\cdot)$ belongs to the so-called $\SG$-classes of order $(2,2)$. 

The $\SG$-calculus appears in various other environments. An invariant definition of the above problems 
can be given on a class of noncompact manifolds, the so-called $\SG$-manifolds, cf. Schrohe \cite{Schro}, which includes the manifolds with ends, see, e.g. 
\cite{CoDo,CoMa,MaPa}. $\SG$-operators are also the local representation of the so-called \textit{scattering operators} on asymptotically Euclidean manifolds, 
see, e.g., Melrose \cite{Me:94}.

The theory of strictly hyperbolic Cauchy problems, for first order systems and PDEs in the $\SG$-setting, can be found in Cordes \cite{Cord}. In particular,
Cordes showed that Cauchy problems of the type \eqref{eq:main}, with smooth coefficients in time within the $\SG$-environment, are well-posed in the Schwartz 
spaces $\scS(\R^{d}), \scSp(\R^{d})$, and in the weighted Sobolev spaces $H^{s,\sigma}(\R^d)$. The latter, also known as Sobolev-Kato spaces, are the scale of 
$L^2$-modelled spaces naturally associated with the $\SG$-calculus, and are defined as
$$
H^{s,\sigma}(\R^d)=\{u\in \scS^\prime({\R^{d}}):\Op(\omega_{s,\sigma})u\in L^{2}(\R^{d})\},
$$
where $\omega_{s,\sigma}(x,\xi)=\norm{x}^{s}\norm{\xi}^{\sigma}$.

Weakly hyperbolic first order
systems and PDEs with constant multiplicities in this same setting and their wellposedness 
have been originally studied by Coriasco \cite{Coriasco}, together with the associated theory of
Fourier integral operators \cite{Coriasco:998.1}, and subsequently by Coriasco and Rodino \cite{Coriasco1}. 
 
Classes of Fourier integral operators globally defined on $\R^d$, involving amplitudes and/or phase functions of $\SG$-type
have been considered, e.g., by Cappiello \cite{Ca04}, Cordero, Nicola and Rodino \cite{CNR}, Coriasco and Ruzhansky \cite{CoRu}, Ruzhansky and Sugimoto \cite{Ruzhansky_Sugimoto}, and others (see the reference lists
of the quoted papers). 
One novelty of our approach here is that we refine the calculus of standard $\SG$-Fourier integral operators, taking into account the subdivision into zones of the phase
space $\R^{2d}$ implied by the symbol class we need to use, in view of the degeneracy of the characteristic roots at $t=0$.

We recall that many authors have considered hyperbolic equations with non-smooth coefficients. In this direction, it has been observed that Lipschitz conditions play a crucial role in the Sobolev regularity results. Colombini and Lerner \cite{CL} have shown that Log-Lipschitz conditions on the coefficient are optimal for Sobolev regularity. Whenever dealing with Log-Lipschitz coefficients, a finite loss of derivatives occurs for the solution of the Cauchy problem as shown by means of examples in \cite{Cicognani, Cicognani1, C, CL}. Other forms of weakening of the Lipschitz condition have been studied extensively in the works of Kubo and Reissig \cite{KR},  Colombini, Del Santo and Reissig \cite{Rei1}, 
and Hirosawa and Reissig \cite{HirReissig2006}, see also Ghisi and Gobbino \cite{GhisiGobbino2021} for cases of finite and infinite loss of derivatives. 
In \cite{AC}, Ascanelli and Cappiello considered $\SG$-hyperbolic models with coefficients satisfying Log-Lipschitz conditions. 
Pattar and Kiran have studied strictly hyperbolic Cauchy problems on $\R^n$ with unbounded and singular coefficients in \cite{PaKi}. 
Extensions of our analysis in similar directions will be the subject of forthcoming papers.

In this paper, we study in detail a second order weakly $\SG$-hyperbolic operator,
in the case where the symbols satisfy a weakened form of Lipschitz condition, encoded through the function $\Sigma(t)$ appearing in \eqref{eq:sym_inequality} above. 
In the procedure for determining the parametrix, we first reformulate the original Cauchy problem into a corresponding Cauchy problem for a first order $2\times2$ system, modulo smoothing elements. We then perform a diagonalization of the $2\times 2$ system, in the spirit of the procedures by 
Yagdjian \cite{Yagdjian}, and Kubo and Reissig \cite{KR} (see also Kumano-go \cite{Kumano-go}). 
The latter requires suitable extensions of some of the results coming from the local symbolic calculi, along the procedures
considered, for instance, 
in \cite{Coriasco, Coriasco1}, which need to be carefully modified. Efficient tools to achieve such needed extensions come, in particular, by some of the results 
by Coriasco and Toft \cite{CoTo}. A special feature of the $\SG$-setting is that one can expect a finite loss of decay along with a loss of derivatives, as observed, e.g., in 
\cite{AC,Coriasco,Coriasco1}. This happens indeed, and we show that within our main results. Notice that the analysis can be carried through for $\SG$-hyperbolic operators of arbitrary 
order, fully combining the mentioned theories in \cite{Coriasco:998.1,Coriasco,Coriasco1,CoTo} with the approaches in \cite{KR,Yagdjian}. 
To keep this exposition within a reasonable length, here we prefer to focus on second order operators, which anyway allows to highlight the main new features apparing in the
$\SG$-environmnt.

The paper is organized as follows. In Section \ref{sec:CP} we recall known facts about the $SG$-calculus, and 
describe our symbolic setting and the assumptions about the coefficients of the operator in \eqref{eq:main}. One first main result, achieved in this section, is the proof of 
an analog, in our setting, of an equivalency between hyperbolicity conditions, in term of properties of the characteristic roots and of the coefficients of the operator.
This also indicates the $\SG$-type symbol classes we then need to consider in our analysis, as well as the subdivision into zones of the phase space. We here state the
properties of such symbol classes. In the subsequent Section \ref{sec:SGpsidosFIOs} we focus on the associated classes of pseudodfferential and Fourier integral operators,
and establish their calculus. Equipped with these two tools, we then follow the classical approach to solve \eqref{eq:main}: in Section \ref{sec:diag} we switch 
to a Cauchy problem for a $2\times2$ first order system, and diagonalize it, determine the parametrix of the diagonalized system in Section \ref{sec:param}, and finally
state and prove our main results about the properties of the solutions to \eqref{eq:main} in the concluding Section \ref{solution_cauchy_problem}. The parametrix
construction relies on the properties of the solutions of the Hamiltonian systems associated with the characteristic roots of the operators. In this respect, a careful analysis
is needed, in view of the explicit presence of the $x$ variable in the definition of the zones, as well as the unboundedness of the characteristic roots with respect to it
and their behaviour with respect to the time variable.
In the sequel we will sometimes
write $A\lesssim B$ when $A\le cB$ for a suitable constant $c > 0$, and we set $A\asymp B$ when $A\lesssim B\lesssim A$ 
(in particular, we will adopt this notation when the value of the specific constants $c,c'$ is not crucial).

\section*{Acknowledgements} 
The first author has been partially supported by his own INdAM GNAMPA Project, Grant Code CUP\underline{\hspace{3mm}}E53C22001930001. 
The authors wish to thank Prof. M. Reissig and Prof. K. Yagdjian, 
for useful hints, comments and discussions. 

\section{Formulation of the hyperbolic Cauchy problem in $\mathbb{R}^{d}$} \label{sec:CP}
\setcounter{equation}{0}
%
%
%

In this section we discuss an appropriate parameter-dependent global symbol class of $\SG$ type, associated with the Cauchy problem \eqref{eq:main}.  
In particular, we will state the requirements on the coefficients that guarantee the $\SG$-hyperbolicity  of \eqref{eq:main} 
(that is, hyperbolicity in $\mathbb{R}^{d}$ with respect to 
the calculus associated with the metric \eqref{SG_metric}), and necessary conditions on the lower order terms that ensure its wellposedness. 
These conditions also dictates the definition of the $\SG$-classes which must be used in the construction of the fundamental solution.

\subsection{Shape function}\label{subs:lambda}
Following \cite{Yagdjian}, we introduce a real-valued, positive function $\lambda(t)$, which allows to describe the speed at which characteristics collide at $t=0$ and the qualitative 
behavior with respect to $t$ of the coefficients in \eqref{eq:main}. In particular,  $\lambda(t)$ belongs to $C^{\infty}([0,T])$ and satisfies $\lambda(0)=\lambda'(0)=0$, 
$\lambda'(t)\neq 0$  for $t>0$. The hypotheses on $\lambda$ of course imply $\lambda'(t)>0$ for $t>0$. Moreover, the following inequality holds
$$
|\lambda^{(k)}(t)|\leq c\left(\frac{\lambda'(t)}{\lambda(t)}\right)^{k-1}|\lambda'(t)|,
$$
for non-negative integers $k$. The choice of $\lambda(t)$ with $\lambda(0)=0$ makes \eqref{eq:main} a multiple characteristic partial differential equation at $t=0$, and the whole
operator one with characteristic roots of variable multiplicities. This function quantifies the vanishing order of the principal part of \eqref{eq:main} with respect to $t$, and thus the behavior of the characteristic roots. Furthermore, the integral of $\lambda(t)$, 
\[
\Lambda(t)=\int_{0}^{t}\lambda(s)ds.
\]
plays a crucial role in the determination of the phase function in the Fourier integral operators we will deal with. We assume that $\lambda^2/\Lambda \in C^{\infty}([0,T])$ and there exist $c_1>1/2$ and $0<C_1<1$ constants such that 
\begin{equation}
\label{eq:lambda'_control}
c_1\frac{\lambda(t)}{\Lambda(t)}\leq \frac{\lambda'(t)}{\lambda(t)}\leq  C_1\frac{\lambda(t)}{\Lambda(t)},
\end{equation} 
for $t\in (0,T]$.
The estimates \eqref{eq:lambda'_control} define a bound on the growth of the coefficients 
encoded by the symbols $a(t,x,\xi)$, given in \eqref{eq:amplitude_a}, with respect to $t$.  
Typical examples of shape functions that satisfy the required assumptions are given by
\begin{equation*}
\lambda(t)=t^{r}\ \quad \textnormal{  or  }\quad \lambda(t)=\underbrace{\exp(-\exp(...-\exp(-|t|^{r})))}_{k \textnormal{ exponents}}
\end{equation*}
for $r,k$ integers numbers, $k\geq 0$ and $r\geq 2$. 

 \subsection{Zones of the phase space}\label{subs:zones}
Employing the function $\Lambda(t)$ we define a partition of $[0,T]\times \R^{2d}$: for all $(x,\xi)\in \R^{2d}$ we fix $t_{x,\xi}$ the unique solution to equation
\begin{equation}\label{txxi_equation}
	\Lambda(t_{x,\xi})\,\norm{x}\norm{\xi} =N\ln\norm{x}\norm{\xi},\quad |x|+|\xi|\geq M,
\end{equation}
where $N>0$ is a sufficiently large parameter and $M>0$. Then, we define the \textit{hyperbolic zone} as
$$
Z_\hyp(N)=\{(t,x,\xi)\in [0,T]\times\mathbb{R}^{2d}:t\geq t_{x,\xi}\},
$$
and the \textit{pseudodifferential zone} as
$$
Z_\pd(N)=\{(t,x,\xi)\in [0,T]\times\mathbb{R}^{2d}:t\leq t_{x,\xi}\}.
$$
For future aims, we further partition the hyperbolic zone into 
as 
$$ 
Z_\hyp(N)=Z_\osc(N)\cup Z_\reg(N),
$$ 
defining the \textit{regular zone} as
$$
Z_\reg(N)=\{(t,x,\xi)\in [0,T]\times\mathbb{R}^{2d}:t_{x,\xi}'\leq t\},
$$ 
and the \textit{oscillation zone} as
$$
Z_\osc(N)=\{(t,x,\xi)\in [0,T]\times\mathbb{R}^{2d}:t_{x,\xi}\leq t\leq t_{x,\xi}'\}.
$$
Here, for all $(x,\xi)\in \R^{2d}$, the time function $t=t_{x,\xi}'$ satisfies the equation 
\begin{equation}\label{regular_zone}
	\Lambda(t_{x,\xi}')\,\norm{x}\norm{\xi}=2N(\ln \norm{x}\norm{\xi})^2 \textnormal{  for  } |x|+|\xi|\geq M.
\end{equation} 
The next Lemma \ref{lem:log_lambda} holds true.
\begin{lem}
\label{lem:log_lambda}
For all $M$ and $N$ sufficiently large, there exist $d_1,d_2>0$ such that it holds
\begin{equation*}
-d_1\ln (\norm{x}\norm{\xi})\leq \ln \lambda(t_{x,\xi}) \leq -d_2\ln (\norm{x}\norm{\xi}),
\end{equation*}
for every $(t,x,\xi)\in Z_{\hyp}(N)$ with $|x|+|\xi|\ge M$. As a consequence, for all $t\in [0,T]$ each couple $(x_t,\xi_t)$ that solves the equation
\begin{equation}\label{eq:timefunbis}
	\Lambda(t)\,\norm{x_t}\norm{\xi_t} =N\ln\norm{x_t}\norm{\xi_t},
\end{equation}
and $|x_t|+|\xi_t|\ge M$ satisfies the inequalities
\begin{equation*} 
	-d_1\ln (\norm{x_{t}}\norm{\xi_{t}})\leq \ln \lambda(t)\leq -d_2\ln (\norm{x_{t}}\norm{\xi_{t}}).
\end{equation*}
\end{lem}
\begin{proof}
We of course have
\[ -\ln\lambda(t_{x,\xi})+\ln\lambda(T)=\int_{t_{x,\xi}}^T \frac{\lambda'(t)}{\lambda(t)}\,dt.\]
By assumption \eqref{eq:lambda'_control}, there exists $c_1,C_1>0$ such that
\[ c_1\int_{t_{x,\xi}}^T \frac{\lambda(t)}{\Lambda(t)}\,dt\leq -\ln\lambda(t_{x,\xi})+\ln\lambda(T)\leq C_1 \int_{t_{x,\xi}}^T \frac{\lambda(t)}{\Lambda(t)}\,dt.\]
Employing \eqref{txxi_equation}, it follows
\begin{align*} 
c_1\ln[N\ln&(\<x\>\<\xi\>)]-c_1\ln(\<x\>\<\xi\>)-c_1\ln\Lambda(T)+\ln\lambda(T) %
\\ & \geq \ln\lambda(t_{x,\xi}) \geq \\ & 
-C_1\ln(\<x\>\<\xi\>)+C_1\ln[N\ln(\<x\>\<\xi\>)]-C_1 \ln\Lambda(T)+\ln\lambda(T).
\end{align*}
For any $N$ (arbitrarily large), and $|x|+|\xi|\ge M$, $M$ sufficiently large, we may guarantee 
\[ -d_1\ln (\norm{x}\norm{\xi})\leq \ln \lambda(t_{x,\xi}) \leq -d_2\ln (\norm{x}\norm{\xi}). \]
The other part of the statement follows by similar considerations, recalling the definition of $Z_{\hyp}(N)$.
\end{proof}
\begin{rem} 
\label{rem:log_lambda}\begin{enumerate}
\item Lemma \ref{lem:log_lambda} implies that, for every $(t,x,\xi)\in Z_{\hyp}(N)$ with $|x|+|\xi|\ge M$, 
$$
d_2\ln (\norm{x}\norm{\xi})\leq |\ln \lambda(t_{x,\xi})| \leq d_1\ln (\norm{x}\norm{\xi})
$$
and, for all $t\in [0,T]$ each couple $(x_t,\xi_t)$ which solves \eqref{eq:timefunbis},
$$
d_2\ln (\norm{x_{t}}\norm{\xi_{t}})\leq |\ln \lambda(t)|\leq d_1\ln (\norm{x_{t}}\norm{\xi_{t}}).
$$
\item As a consequence of Lemma \ref{lem:log_lambda}, there exists $C>0$ such that
\begin{equation*}
\label{ineq_lambda}
	\norm{x_{t}}\norm{\xi_{t}}=\frac{N(\ln \norm{x_{t}}\norm{\xi_{t}})}{\Lambda(t)}\leq C\frac{|\ln \lambda(t)|}{\Lambda(t)}.
\end{equation*}
Moreover, for $M$ sufficiently large, in the hyperbolic zone we may estimate $$\frac{|\ln \lambda (t)|}{\norm{x}\norm{\xi}\Lambda(t)}\lesssim \frac{1}{N}.$$
Indeed, we first prove that for all $t\in [0,T]$ and $(x,\xi)\in \R^{2d}$, if $(t,x,\xi)\in Z_{\hyp}(N)$, then $\<x\>\<\xi\>\geq \<x_t\>\<\xi_t\>$. In fact, assume
that $\<x\>\<\xi\><\<x_t\>\<\xi_t\>$. Then, since the function $\zeta(s)=\frac{s}{\ln(s)}$ is increasing on $[e,+\infty)$, we have
\[ \frac{\Lambda(t)\<x\>\<\xi\>}{\ln(\<x\>\<\xi\>)}< \frac{\Lambda(t)\<x_t\>\<\xi_t\>}{\ln(\<x_t\>\<\xi_t\>)}=N,\]
that is, $(t,x,\xi)\in Z_\pd(N)$, which is a contradiction. It follows, by point (1) and the previous considerations,
\[ \frac{1}{\Lambda(t)}=\frac{\<x_t\>\<\xi_t\>}{N\ln(\<x_t\>\<\xi_t\>)}\leq \frac{d_1\<x_t\>\<\xi_t\>}{N|\ln(\lambda(t))|}\leq \frac{d_1\<x\>\<\xi\>}{N|\ln(\lambda(t))|},\]
which gives the claim.
\end{enumerate}
\end{rem}

\subsection{Assumptions on the coefficients}\label{subs:coeff} We suppose that there exist a positive constant $N$ such that the zeros $\tau_j(t,x,\xi)$, $j=1,2$, 
of the complete symbol of the operator $L$,
\begin{equation}
\label{eq:complete_symbol}
\mathfrak{L}(t,\tau,x,\xi)=-\tau^2+a(t,x,\xi)=0,
\end{equation}
are smooth functions on $Z_\hyp(N)$ and, for any $k\in\N$, $\alpha,\beta\in\N^d$, the following inequalities hold:
\begin{align}
\label{eq:tau_estimate}
|D_t^kD_x^\alpha D_\xi^\beta \tau_j(t,x,\xi)|&\leq C_{k\alpha\beta}\lambda(t)\<x\>^{1-|\alpha|}\<\xi\>^{1-|\beta|}\bigg(\frac{\lambda(t)}{\Lambda(t)}\ln\bigg(\frac{1}{\Lambda(t)}\bigg)\bigg)^k,\\
\label{eq:tau_hyperbolicity}
|\tau_1(t,x,\xi)-\tau_2(t,x,\xi)|&\geq \delta \lambda(t)\<x\>\<\xi\>, \\
\label{eq:tau_imaginarypart}
| D_t^k D_x^\alpha D_\xi^\beta \text{Im} \tau_j(t,x,\xi)|&\leq C_{k\alpha\beta}\<x\>^{-|\alpha|}\<\xi\>^{-|\beta|}\bigg(\frac{\lambda(t)}{\Lambda(t)}\ln\bigg(\frac{1}{\Lambda(t)}\bigg)\bigg)^{k+1},
\end{align}
for some positive constants $\delta$, $C_{k\alpha\beta}$, and $(t,x,\xi)\in Z_\hyp(N)$.
Proposition \ref{prop: equivalence} below allows to relate the conditions on the roots $\tau_j(t,x,\xi)$, $j=1,2$, of the complete symbol with suitable assumptions on the roots $\lambda_j(t,x,\xi)$ of the principal symbol of the operator $L$,
\begin{equation}
\label{eq:principal_symbol}
\mathfrak{L}_p(t,\tau,x,\xi):=-\tau^2+\sum_{j=1}^d a_j(t,x)\xi_j^2,
\end{equation}
and on the coefficients $a_j(t,x)$, $b_j(t,x)$, $j=1,\dots, d$ and $c(t,x)$.
Such equivalent formulation of $\SG$-hyperbolicity of $L$ is proved by adapting the argument of the analogous statement in \cite{Yagdjian}.
We sketch the argument below, focusing on the specific aspect involving here the unboundedness of the coefficients, characteristic roots, and symbols,
with respect to the space variable $x\in\R^d$.
\begin{prop}
\label{prop: equivalence}
The following conditions \textbf{(A)} and \textbf{(H)} are equivalent:
\begin{itemize}
\item[\textbf{(A)}] For all $t\in [0,T]$, $x,\xi\in \R^d$, the zeros $\lambda_1(t,x,\xi)$ and $\lambda_2(t,x,\xi)$ of the principal symbol 
\eqref{eq:principal_symbol} of the operator $L$ are real-valued functions that satisfy the following inequalities:
\begin{align}
\label{eq:lambda_estimate}
|\lambda_j(t,x,\xi)|&\leq c \lambda(t)|\xi| \<x\>, \qquad j=1,2, \\
\label{eq:lambda_hyperbolicity}
|\lambda_1(t,x,\xi)-\lambda_2(t,x,\xi)|&\geq \delta_1 \lambda(t)|\xi|\<x\>,
\end{align}
for some positive constant $c$ and $\delta$, independent of $t\in [0,T]$, $(x,\xi)\in \R^{2d}$. Furthermore, the coefficients $a_j(t,x)$, $b_j(t,x)$, $j=1,\dots, d$, and $c(t,x)$ satisfy the inequalities:
\begin{equation}
\label{eq:coefficients_assumptions}
\begin{aligned}
|D_t^k D_x^\alpha a_j(t,x)|&\leq C_{k\alpha} \lambda(t)^2 \<x\>^{2-|\alpha|} \bigg(\frac{\lambda(t)}{\Lambda(t)}\ln\bigg(\frac{1}{\Lambda(t)}\bigg)\bigg) ^k, \\
|D_t^k D_x^\alpha b_j(t,x)|&\leq C_{k\alpha}\lambda(t)^2 \<x\>^{1-|\alpha|} \frac{|\ln\lambda(t)|}{\Lambda(t)}\bigg(\frac{\lambda(t)}{\Lambda(t)}\ln\bigg(\frac{1}{\Lambda(t)}\bigg)\bigg) ^k,\\
|D_t^k D_x^\alpha c(t,x)|&\leq C_{k\alpha}\lambda(t)^2 \<x\>^{-|\alpha|} \bigg(\frac{\ln\lambda(t)}{\Lambda(t)}\bigg)^2 \bigg(\frac{\lambda(t)}{\Lambda(t)}\ln\bigg(\frac{1}{\Lambda(t)}\bigg)\bigg) ^k,\\
|D_t^k D_x^\alpha \mathrm{Im} b_j(t,x)|&\leq C_{k\alpha}\lambda(t)\<x\>^{1-|\alpha|}\bigg(\frac{\lambda(t)}{\Lambda(t)}\ln\bigg(\frac{1}{\Lambda(t)}\bigg)\bigg) ^{k+1},
\end{aligned}
\end{equation}
for some $C_{k\alpha}>0$, independent of $t\in (0,T]$ and all $x\in \R^d$. Additionally, the coefficients $a_j$, $j=1,\dots,d$, are real-valued functions.
\item[\textbf{(H)}] There exist a positive constant $N_0$ such that, for all $N\geq N_0$, the zeros $\tau_1(t,x,\xi)$ and $\tau_2(t,x,\xi)$ of the 
complete symbol \eqref{eq:complete_symbol} of the operator $L$ are smooth functions on $Z_\hyp(N)$ and the inequalities 
\eqref{eq:tau_estimate}-\eqref{eq:tau_imaginarypart} hold for all $(t,x,\xi)\in Z_\hyp(N)$.
\end{itemize}
\end{prop}
\begin{proof}
We prove the implication $\textbf{(H)}\implies \textbf{(A)}$.  Since it holds 
\[ a_j(t,x)= \frac{1}{2} D_{\xi_j}^2 (\tau_1(t,x,\xi)\tau_2(t,x,\xi)),\]
for all $(t,x,\xi)\in Z_\hyp(N)$ we may estimate
\begin{equation}
\label{eq:aj_estimate}
|D_t^k D_x^\alpha a_j(t,x)|\leq C_{k\alpha} \lambda(t)^2\<x\>^{2-|\alpha|}\weight.
\end{equation} 
By induction, we may prove the desired estimates also for the coefficients $b_j$, $j=1,\dots, d$ and $c$.
Indeed, it holds
\[ D_t^k D_x^\alpha b_j(t,x)=i D_t^k D_x^\alpha ( D_{\xi_j}(\tau_1\tau_2 )-2a_j(t,x)\xi_j),\]
and 
\[D_t^kD_x^\alpha c(t,x)= D_t^k D_x^\alpha \Big(\tau_1\tau_2-\sum_{j=1}^d (a_j(t,x)\xi_j^2+b_j(t,x)\xi_j)\Big);\]
In particular, by \eqref{eq:tau_estimate} and \eqref{eq:aj_estimate} we have
\begin{align*} 
|D_t^kD_x^\alpha b_j(t,x)|\leq &C_{k\alpha}\lambda(t)^2\<x\>^{2-|\alpha|}\<\xi\> \weight \\\leq &C_{k\alpha} \lambda(t)^2 \<x\>^{1-|\alpha|}\frac{|\ln\lambda(t)|}{N\Lambda(t)}\weight,
\end{align*}
and then,
\begin{align*} 
	|D_t^kD_x^\alpha c(t,x)|\leq &C_{k\alpha}\lambda(t)^2\<x\>^{2-|\alpha|}\<\xi\>^2\weight \\\leq &C_{k\alpha} \lambda(t)^2\<x\>^{-|\alpha|} \bigg(\frac{|\ln\lambda(t)|}{N\Lambda(t)}\bigg)^2\weight,
\end{align*}
for all $(t,x)\in [0,T]\times \R^d$. Indeed, for any fixed $t\in [0,T]$ and $x\in \R^d$, it is possible to choose $\xi\in \R^d$ (depending on $t$ and $x$)  equal to the unique solution to the equation
\[ \Lambda(t)\<x\>\<\xi\>=N\ln(\<x\>\<\xi\>).\]
With such choice of $\xi$, we may estimate 
\[ \<x\>\<\xi\>\leq C \frac{|\ln(\lambda(t))|}{N\Lambda(t)},\]
as a consequence of Lemma \ref{lem:log_lambda} (see Remark \ref{rem:log_lambda}).\\
Let us now prove that $\textrm{Im} a_j=0$. Applying estimate \eqref{eq:tau_imaginarypart} we find
\begin{align*}
| \text{Im}\,a_j(t,x)|&=\big|\text{Im}\,\partial_{\xi_j}^2 (\tau_1(t,x,\xi)\tau_2(t,x,\xi))\big|\\
&\lesssim \<x\> \<\xi\>^{-1}\frac{\lambda(t)^2}{\Lambda(t)}\ln\bigg(\frac{1}{\Lambda(t)}\bigg),
\end{align*}
for all $(t,x,\xi)\in Z_\hyp(N)$. In particular, for any fixed $t>0$ and $x\in \R^n$ there exists $M$ sufficiently large such that $(t,x,\xi)\in Z_\hyp(N)$ for all $|\xi|>M$. Taking $|\xi|\to +\infty$ the right-hand side tends to $0$, whereas the left-hand side does not depend on $\xi$. This allows to conclude that $\text{Im}a_j=0$. As a consequence, the following identity holds:
\begin{align*}
	| \text{Im}\,b_j(t,x)|&=\big|\text{Im}\,\partial_{\xi_j} (\tau_1(t,x,\xi)\tau_2(t,x,\xi))\big|\\
	&\lesssim \<x\> \frac{\lambda(t)^2}{\Lambda(t)}\ln\bigg(\frac{1}{\Lambda(t)}\bigg).
\end{align*}
In order to prove \eqref{eq:lambda_estimate} we define, for $i=1,2$, $\mu_i(t,x,\xi)$ and $\gamma_i(t,x,\xi)$ such that 
$$
\lambda_i(t,x,\xi)=\lambda(t)\<x\>|\xi|\mu_i(t,x,\xi) \text{ and } \tau_i(t,x,\xi)=\lambda(t)\<x\>|\xi|\gamma_i(t,x,\xi).
$$
Then, $\mu_i$ and $\gamma_i$ satisfy
\begin{align}
& \label{eq:mui_equation} -\mu_i^2 + (\lambda(t)\<x\>|\xi|)^{-2} \sum_{j=1}^d a_j(t,x)\xi_j^2 =0; \\
& \label{eq:gammai_equation}-\gamma_i^2 + (\lambda(t)\<x\>|\xi|)^{-2} \sum_{j=1}^d  a_j(t,x)\xi_j^2 + B_1, \\  &\hspace{50pt} \notag B_1:= (\lambda(t)\<x\>|\xi|)^{-2} \Big(\sum_{j=1}^d  b_j(t,x)\xi_j +c(t,x)\Big)=0.
\end{align}
In view of conditions \eqref{eq:tau_estimate}-\eqref{eq:tau_imaginarypart}, the functions $\gamma_i$ satisfy:
\begin{align}
	\label{eq:gamma_estimate}
	|D_t^jD_x^\alpha D_\xi^\beta \gamma_i(t,x,\xi)|&\leq C_{k\alpha\beta}\<x\>^{-|\alpha|}\<\xi\>^{-\beta}\bigg(\frac{\lambda(t)}{\Lambda(t)}\ln\bigg(\frac{1}{\Lambda(t)}\bigg)\bigg)^j,\\
	\label{eq:gamma_hyperbolicity}
	|\gamma_1(t,x,\xi)-\gamma_2(t,x,\xi)|&\geq \delta, \\
	%
\notag	| D_t^j D_x^\alpha D_\xi^\beta \text{Im} \gamma_i(t,x,\xi)|&\leq C_{k\alpha\beta}\frac{1}{\Lambda(t)}\<x\>^{-1-|\alpha|}\<\xi\>^{-1-|\beta|}\bigg(\frac{\lambda(t)}{\Lambda(t)}\ln\bigg(\frac{1}{\Lambda(t)}\bigg)\bigg)^{j},
\end{align}
for all $(t,x,\xi)\in Z_\hyp(N)$. Furthermore, by using the obtained estimates for the coefficients $b_j$, $j=1,\dots,d$, and $c$ we may estimate the perturbation $B_1$ by
\begin{equation} 
\label{eq:B1}
|D_t^k D_x^\alpha D_\xi^\beta B_1|\lesssim  \frac{1}{N}\<x\>^{-|\alpha|}\<\xi\>^{-|\beta|}\weight.
\end{equation}
Let us consider the polynomial in the variable $\mu$ defined by $P(t,x,\xi; \mu):=(\mu-\gamma_1)(\mu-\gamma_2)$. By \eqref{eq:mui_equation} and \eqref{eq:gammai_equation} we derive that $\mu_1$ and $\mu_2$ are the two solutions to equation
\begin{equation*}
P(t,x,\xi;\mu)+B_1=0.
\end{equation*}
In particular, condition \eqref{eq:gamma_hyperbolicity} guarantees that $P(t,x,\xi, \cdot)$ has two simple roots. Moreover, the roots $\gamma_i$ of \eqref{eq:gammai_equation} analytically depend on the perturbation $B_1$ in some neighborhood of the origin. Then, if $B_1$ has modulus sufficiently small we may write the following series expansion of $\mu_i(t,x,\xi)$,
\[ \mu_i(t,x,\xi)=\gamma_i(t,x,\xi)+\sum_{n=1}^\infty c_n^{(i)}(t,x,\xi) B_1^n(t,x,\xi), \quad i=1,2,\]
where 
\begin{align*}
c_n^{(i)}(t,x,\xi) &= \frac{1}{2\pi i }\oint_{|\omega-\gamma_i|=\rho} \frac{(\omega-\omega_i)P'_\omega(t,x,\xi;\omega)}{P(t,x,\xi;\omega)^{n+1}}\\
& = \frac{1}{(n-1)!}\left[\frac{d^{n-1}}{d\omega^{n-1}}\left\{\bigg[\frac{\omega-\gamma_i}{P(t,x,\xi;\omega)}\bigg]^{n+1}P'_\omega(t,x,\xi;\omega)\right\}\right]_{\omega=\gamma_i}.
\end{align*}
where $0<2\rho<\delta$. In particular, we have the inequality 
\[c_n^{(i)}(t,x,\xi)\leq c \delta^{-1}(2c/\rho\delta)^n,\]
for all $(t,x,\xi) \in Z_\hyp(N)$, with a constant $c$ independent of $t$, $x$, $\xi$. As a consequence, the radius of convergence given by the Cauchy-Hadamard formula $\displaystyle{R_i=1/\limsup_{n\to +\infty} (c_n^{(i)} )^\frac{1}{n}}$ is independent of $(t,x,\xi)$.
Moreover, due to estimate \eqref{eq:B1}, the series $\sum_{n} c_n^{(i)}B_i^n$ can be made arbitrarily small taking $N$ sufficiently large.
Then, by estimate \eqref{eq:B1}, for $N$ sufficiently large we derive 
\begin{equation}
\label{eq:mui_hyperbolicity}
|\mu_1(t,x,\xi)-\mu_2(t,x,\xi)|\geq \delta 
\end{equation}
for all $(t,x,\xi)\in Z_\hyp(N)$. Let us prove the following estimate:
\begin{equation}
	\label{eq:mui_estimate}
|D_t^k D_x^\alpha D_\xi^\beta \mu_i(t,x,\xi)|\leq C_{k\alpha\beta} \<x\>^{-|\alpha|}\<\xi\>^{-|\beta|}\bigg(\frac{\lambda(t)}{\Lambda(t)}\ln\bigg(\frac{1}{\Lambda(t)}\bigg)\bigg)^k,
\end{equation}
for all $(t,x,\xi)\in Z_\hyp(N)$. To this aim, let us denote $y=(t,x,\xi)\in (0,T]\times \R^{2d}$. By applying the formula for the derivative of implicit functions to $P(t,x,\xi; \mu)$, for any multi-index $r\in \N^{2d+1}$ and $i=1,2$, we obtain the identity
\begin{align*}
\partial_y^r \mu_i(y)= (P'_\omega(y;\mu_i(y)))^{-1}(-B_1(y) +\mathcal{E}_i^r(y)),
\end{align*}
where
\begin{align*}
\mathcal{E}_i^r(y)&=\partial_y^r \gamma_1(y)(\mu_i(y)-\gamma_2(y))+\partial_y^r \gamma_2(y)(\mu_i(y)-\gamma_1(y))\\
&-\sum_{\substack{ r_1+r_2=r \\ r_1, r_2 \neq r}} \frac{r!}{r_1! r_2!} \partial_y^{r_1}(\mu_i(y)-\gamma_1(y))\partial_y^{r_2}(\mu_i(y)-\gamma_2(y)).
\end{align*}
By \eqref{eq:B1} and \eqref{eq:mui_hyperbolicity} we may estimate, for $N$ sufficiently large,
 $$
 |P'_\omega(y,\mu_i(y))|=|2\mu_i-(\gamma_1+\gamma_2)|> \delta
 $$
 for all $y\in Z_\hyp(N)$. Then, \eqref{eq:mui_estimate} follows from \eqref{eq:gamma_estimate} and our previous considerations on $\sum_{n} c_n^{(i)}B_i^n$.
\end{proof}
\subsection{Generalized parameter-dependent \textbf{SG} symbol classes}\label{subs:sgclasses}
Following \cite{CoTo}, we consider, for any real numbers $m,\mu$, and $r_j,\rho_j\ge0$, $j=1,2$, the general class of $\SG$-type symbols 
$\SGs^{m,\mu}_{r_1,r_2,\rho_1,\rho_2}$, which consists of all $a\in C^\infty(\R^{2d})$ such that, for all multi-indices $\alpha,\beta\in\N^d$, 
there is a constant $C_{\alpha\beta}>0$ that satisfies
\begin{equation*}
	|D^{\alpha}_{x}D_{\xi}^{\beta}a(x,\xi)|\leq C_{\alpha\beta}\norm{x}^{m-r_1 |\alpha|+r_2 |\beta|}\norm{\xi}^{\mu+\rho_1|\alpha|-\rho_2|\beta|}, 
\end{equation*}
for all $(x,\xi)\in \R^{2d}$. In particular, we denote the special classes $\SGs^{m,\mu}_{1,0,1,0}$ and $\SGs^{m,\mu}_{1-\ve,\ve,1-\ve,\ve}$, $\ve\in(0,1)$, by 
$\SGs^{m,\mu}$ and $\SGs^{m,\mu}_{(\ve)}$, respectively. 
\begin{defn}
\label{def: SG-ellipticity}
An element $a\in\SGs^{m,\mu}$ is called ($md$- or $\SG$-)elliptic if it also satisfies the lower bound
$$
|a(x,\xi)| \ge C\norm{x}^m\norm{\xi}^\mu, \quad x,\xi\in\R^d, |x|+|\xi|\ge M.
$$
More generally, $a\in\SGs^{m,\mu}$ is called ($\SG$-)hypoelliptic if it satisfies the lower bound
$$
|a(x,\xi)| \ge C\norm{x}^{m'}\norm{\xi}^{\mu'}, \quad x,\xi\in\R^d, |x|+|\xi|\ge M,
$$
for some $m'\le m$, $\mu'\le\mu$, and for any $\alpha,\beta\in \N^d$ there exists $C_{\alpha\beta}>0$ such that 
\[ |D_x^\alpha D_\xi^\beta a(x,\xi)|\leq C|a(x,\xi)|\<x\>^{-|\alpha|}\<\xi\>^{-|\beta|}.\]
\end{defn}

Note that an extended family of $\SG$-symbols is defined in \cite{CJT2}, by employing more general weights $\omega(x,\xi)$ in place of 
$\langle x \rangle^{m}\langle \xi \rangle^{\mu}$ in \eqref{eq:generalSGclass}. 
\begin{defn}
	By $\SGP{2N}$ we denote the class of all symbols 
	$p \in L_\infty\Big([0,T],$ $C^\infty(\R^{2d})\Big)$ satisfying, for $(t,x,\xi) \in Z_\pd(2N)$  and for all multi-indices $\alpha, \beta\in\N^d$, 
	the estimates 
	$$
	\esssup_{t \in [0,t_{x,\xi}]}
	|\pa_\xi^\beta \pa_x^\al p(t,x,\xi)| \le C_{\alpha\beta} \,  \norm{x}^{1-|\al|}\norm{\xi}^{1-|\beta|}.
	$$
\end{defn}


\begin{defn}\label{symbol_class_hyp}
	By $\SGH{m}{\mu}{\kappa}{\ell}{N}$ we denote the class of all symbol families $a\in C([0,T], S^{m,\mu})\cap C^{\infty}((0,T]\times \mathbb{R}^{2d})$ such that,
	for all $k\in \N$ and $\alpha,\beta\in\N^d$, we have, for suitable constants $C_{k\alpha\beta}>0$, 
	\begin{equation}
	\label{eq:estimates_hypSymbolClass}
		|D_t^kD^{\alpha}_{x}D_{\xi}^{\beta}a(t,x,\xi)|\leq 
		C_{k\alpha\beta}\norm{x}^{m-|\beta|}\norm{\xi}^{\mu-|\alpha|}\lambda(t)^{\kappa}\left(\frac{\lambda(t)}{\Lambda(t)}\left(\ln \frac{1}{\Lambda(t)}\right)\right)^{\ell+k}
	\end{equation}
 	for every $(t,x,\xi)\in Z_\hyp(N)$. Similarly, we denote by $\SGO{m}{\mu}{\kappa}{\ell}{N}$ the class of all symbol families in 
	$C([0,T], S^{m,\mu})\cap C^{\infty}((0,T]\times \mathbb{R}^{2d})$ that satisfy \eqref{eq:estimates_hypSymbolClass}  for all $(t,x,\xi)\in Z_\reg(N)$.
\end{defn} 
\begin{rem}
The class $\SGO{m}{\mu}{\kappa}{\ell}{N}$ has the same properties as that of $\SGH{m}{\mu}{\kappa}{\ell}{N}$ except for regularity behavior. In particular, the regularity in $\SGO{m-l}{\mu-l}{\kappa}{\ell+l}{N}$ is $(\ln\norm{x}\norm{\xi})^{-1}$ better than those of symbols from $\SGO{m-l+1}{\mu-\ell+1}{\kappa}{\ell+l-1}{N}$ for $l\geq 0$. This deviation of the regularity across hierarchy classes will play a role in the diagonalization procedure below. 
\end{rem}
Some hierarchical properties of the symbol class introduced in Definition \ref{symbol_class_hyp} are listed below. They follow by adapting arguments
to prove similar properties of the classes employed in \cite{Yagdjian}, so we omit their proofs.
\begin{prop}\label{prop:hierarchy}
\begin{enumerate}
	\item $\SGH{m}{\mu}{\kappa}{\ell}{N_2}\subset \SGH{m}{\mu}{\kappa}{\ell}{N_1}$ for $N_1\geq N_2$.
	\item $\SGH{m}{\mu}{\kappa}{\ell}{N}\subset\SGH{m+l}{\mu+l}{\kappa}{\ell-l}{N}$ for $l\geq 0$. By this property there is no difference in regularity of the symbols from  $\SGH{m}{\mu}{\kappa}{\ell}{N}$ and $\cap_{l\geq 0}\SGH{m-l}{\mu-l}{\kappa}{\ell+l}{N}$ in $Z_\hyp(N)$. 
	\item If $a(t,x,\xi)\in \SGH{m_1}{\mu_1}{\kappa_1}{\ell_1}{N}$ and $b(t,x,\xi)\in \SGH{m_2}{\mu_2}{\kappa_2}{\ell_2}{N}$, then $a(t,x,\xi)b(t,x,\xi)\in \SGH{m_1+m_2}{\mu_1+\mu_2}{\kappa_1+\kappa_2}{\ell_1+\ell_2}{N}.$ 
	\item  If $a(t,x,\xi)\in \SGH{m}{\mu}{\kappa}{\ell}{N}$ and $\kappa >0$, then $D_{t}a(t,x,\xi)\in \SGH{m}{\mu}{\kappa-1}{\ell+1}{N}$ 
	\item If $a\in \SGH{m}{\mu}{\kappa}{\ell}{N}$ is constant in $Z_\pd(N)$ and $\ell\ge0$, then 
	$$ 
	\partial_{t}^{l} a \in L_\infty \Big([0,T],\; S^{\tilde{m}+l,\tilde{\mu}+l} \Big)
	$$
	for all $l\geq 0$ and where $\tilde{m}= \max\{0,m+\ell\}$ and $\tilde{\mu}= \max\{0,\mu+\ell\}$. Indeed, for $(t,x,\xi) \in Z_\hyp(N)$  we have
	\begin{eqnarray}
	\nonumber |\partial_{t}^{l}a(t,x,\xi)|&\leq& \norm{x}^{m}\norm{\xi}^{\mu}\lambda(t)^\kappa\left(\frac{\lambda(t)}{\Lambda(t)}\ln \left(\frac{1}{\Lambda(t)}\right)\right)^{\ell+l}\\
	\nonumber &\leq&C_{l}\norm{x}^{m+\ell+l}\norm{\xi}^{\mu+\ell+l},
	\end{eqnarray}
	being $\lambda(t)$ uniformly bounded in $[0,T]$ and 
	$$ \frac{\lambda(t)}{\Lambda(t)}\ln\left( \frac{1}{\Lambda(t)}\right)\leq \<x\>\<\xi\>. $$
\end{enumerate}
\end{prop}
A straightforward modification of the procedures of aysmptotic summation, using the hierarchy of symbol classes $ \SGH{m}{\mu}{\kappa}{\ell}{N}$,
yield the next Lemma \ref{lem:asymptotic}.
\begin{lem}\label{lem:asymptotic}
	Assume that the symbols $a_l \in \SGH{m-l}{\mu-l}{\kappa}{\ell}{N},\, \kappa \ge 0$,
	vanish in $Z_\pd(N)$. Then, there is a symbol $a\in \SGH{m}{\mu}{\kappa}{\ell}{N}$ with support in $Z_\hyp(N)$ such that 
	$$
	a-\sum_{l=0}^{k-1}a_{l} \in \SGH{m-k}{\mu-k}{\kappa}{\ell}{N}, \textnormal{ for all }  k\geq 1.
	$$
	The symbol $a$ is uniquely determined modulo $C^{\infty}([0,T],\SGs^{-\infty,-\infty})$.
\end{lem} 
\begin{rem}
\label{rem:a_classes}
We note that if \eqref{eq:tau_estimate} holds for $j=1,2$, for all $(t,x,\xi)$ in $Z_\hyp(N)$, then the function $a$ defined in \eqref{eq:amplitude_a} belongs to $\SGH{2}{2}{2}{0}{N}$.
Moreover, for all $(t,x,\xi)\in Z_{\pd}(N)$ it satisfies 
\[ |\partial_x^\alpha \partial_\xi^\beta a(t,x,\xi)|\lesssim \<x\>^{2-|\alpha|}\<\xi\>^{2-|\beta|}.\] 
Further, inequality \eqref{eq:tau_hyperbolicity} implies that $a(t,x,\xi)$ satisfies a weak ellipticity condition, that is 
\[ a(t,x,\xi)\gtrsim \lambda(t)^2\<x\>^2\<\xi\>^2,\]
for all $(t,x,\xi)\in Z_\hyp(N)$. In particular, $a(0,x,\xi)$ needs not to be $\SG$-elliptic.
\end{rem}
\begin{rem}\label{rem:asymptexpbis}
	Notice that the results in Proposition \ref{prop:hierarchy} and Lemma \ref{lem:asymptotic} hold true, with trivial
	modifications, also for hierarchies associated with the subdivision into zones, and
	based on the more general classes $\SGs^{m,\mu}_{r_1,r_2,\rho_1,\rho_2}$, $r_j,\rho_j\ge0$, $j=1,2$, 
	$r_2,\rho _1<1$, and, in particular, on the classes $\SGs^{m,\mu}_{(\ve)}$, $\ve\in(0,1)$, mentioned above. We will tacitly use these results as well in the sequel,
	wherever needed.
\end{rem}
Now we provide an example of operator belonging to the class we are studying. The Cauchy problem can, in this case, be solved
explicitely, so that we can, in particular, study the decay properties of the solution in relation to those of the initial data, which is one
of the new, interesting features on which we are focusing. 

\begin{exmp}\label{ex:oprnohyp}
Choose a shape function $\lambda$, satisfying the hypotheses described in Section \ref{subs:lambda}, and
consider the partial differential equation
\begin{equation}\label{eq:examplenoloss}
Lu(t,x)= (\partial_{t}-\lambda(t)x\partial_{x})(\partial_{t}+\lambda(t)x\partial_{x})u(t,x)=0, \quad (t,x)\in[0,T]\times\R, 
\end{equation}
with the unknown $u$ satisfying 
$$
u(0,x)=f(x) \textnormal{   and   } u_{t}(0,x)=g(x).
$$
The solution is given by 
\begin{equation}\label{eq:exmplexplsoln}
u(t,x)=f(xe^{-\Lambda(t)})+\int_{0}^{t}g\left(xe^{2\Lambda(s)-\Lambda(t)}\right)ds,
\end{equation}
as it can be checked directly.

By straightforward computations, with the notation in \eqref{eq:amplitude_a}, dropping the index $j\equiv1$, we here have
$$
a(t,x)=\lambda(t)^2x^2, \;
b(t,x)=i(\lambda'(t)-\lambda(t)^2)x, \textnormal{  and  } c(t,x)=0.
$$ 
%
The roots of the principal symbol \eqref{eq:principal_symbol}, $\lambda_{1,2}(t,x,\xi)=\pm\lambda(t)|x\xi|$, are real-valued, satisfy
\eqref{eq:lambda_estimate}, but violate \eqref{eq:lambda_hyperbolicity}. It is also immediate to notice that
\[
	a(t,x,\xi)=\lambda(t)^2x^2\xi^2 + i(\lambda'(t)-\lambda(t)^2)x\xi
\]
does not fulfill \eqref{ellipticity_3}. However, the coefficients satisfy the estimates \eqref{eq:coefficients_assumptions}. 
Indeed, the estimates for $c$ are trivial, while the estimates involving order $0$ derivatives and any $x$-derivatives 
of $a$ and $b$ are immediate, given their product form. So, only the behaviour of the $t$-derivatives needs to be 
checked. To this aim, we employ the properties of $\lambda(t)$, which imply, in particular,
$$
|\lambda'(t)|=\left|\frac{\lambda'(t)}{\lambda(t)}\lambda(t)\right|\leq C_{1} \frac{\lambda^{2}(t)}{\Lambda(t)}. 
$$ 
Then,
\begin{align*}
	|D_t \lambda(t)^2|&=2|\lambda(t)\,\lambda'(t)|\le 2\lambda(t)^2\frac{\lambda(t)}{\Lambda(t)}\lesssim
	\lambda(t)^2\left(\frac{\lambda(t)}{\Lambda(t)}\left(\ln\frac{1}{\Lambda(t)}\right)\right),
	\\
	|D_t^2 \lambda(t)^2|&=|2(\lambda'(t))^2+2\lambda(t)\,\lambda''(t)|
	\\
	&\lesssim\lambda(t)^2\left(\frac{\lambda(t)}{\Lambda(t)}\left(\ln\frac{1}{\Lambda(t)}\right)\right)^2+\lambda(t)\left|\frac{\lambda'(t)}{\lambda(t)}\right|\,|\lambda'(t)|
	\\
	&\lesssim\lambda(t)^2\left(\frac{\lambda(t)}{\Lambda(t)}\left(\ln\frac{1}{\Lambda(t)}\right)\right)^2+\lambda(t)\frac{\lambda(t)}{\Lambda(t)}\lambda(t)\frac{\lambda(t)}{\Lambda(t)}
	\\
	&\lesssim\lambda(t)^2\left(\frac{\lambda(t)}{\Lambda(t)}\left(\ln\frac{1}{\Lambda(t)}\right)\right)^2,
\end{align*}
and the estimates for the higher order derivatives follow by induction. This proves the estimates for the coefficient $a$. The estimates for the coefficient $b$ can be obtained
in a completely similar fashion, taking into account that $\displaystyle\left|\frac{\ln\lambda(t)}{\Lambda(t)}\right|\gtrsim1$.

We conclude the analysis of this example observing that, in this case, the solution \eqref{eq:exmplexplsoln} has the same decay of the initial data (as $|x|\to +\infty$),
in spite of the fact that the operator has characteristics with variable multiplicites (distinct for $t\in(0,T]$, both collapsing to zero at $t=0$). 
This shows that the decay loss phenomenon, which occurs for equation \eqref{eq:main} under Assumption \textbf{(A)} (or, equivalently, Assumption \textbf{(H)})
in Proposition \ref{prop: equivalence}, 
is strongly related to the behaviour of the Hamiltonian flows generated by the characteristic roots, which \textit{transports the smoothness and decay singularities},
encoded by suitable \textit{global wave-front sets}, see \cite{CJT2,CJT3,CoMaSWF} and Section \ref{subs:HamFlow} below. For the operator $L$ in \eqref{eq:examplenoloss}, 
the flow generated by $\lambda_1$ is given, for $x,\xi>0$, by $(x,\xi)\mapsto(x,\xi)\exp(\Lambda(t)-\Lambda(s))$ (and similar
expressions in the other quadrants and for $\lambda_2$), which \textit{preserves directions $\infty x_0$} (see \cite{CJT3,CoMaSWF}), and then \textit{decay singularities} as well.
Results about the \textit{propagation of global singularities} for the Cauchy problems studied in this paper, further extending those in \cite{AAC,CJT3,CoMaSWF}, will 
appear elsewhere.  
\end{exmp}

\section{Generalized pseudodifferential operators and Fourier integral operators of $\SG$-type} \label{sec:SGpsidosFIOs}
\setcounter{equation}{0}
%
%
%

Let $a\in S^{m,\mu}_{r_1,r_2,\rho_1,\rho_2}$ be such that $r_j,\rho_j\ge0$, $j=1,2$, $\rho_1\leq \rho_2$ and $r_1\geq r_2$. 
Then, the pseudodifferential operator associated with $a$, denoted by $\Op(a)$, is a linear and continuous operator on 
$\mathcal{S}(\mathbb{R}^{d})$ defined by the formula 
$$
(\Op(a)f)(x)=(2\pi)^{-d}\int\int e^{i(x-y)\cdot\xi}a(x,\xi)f(y)dy d\xi.
$$
Given a $\SG$-pseudodifferential operator $A$, we will, as customary, denote by $\sigma(A)$ its symbol. We will also denote by $\sigma_p(A)$ a principal 
part of the symbol of $A$ (often, this will be the leading term of an asymptotic expansion). Explicitly, if $A\in \Op(S^{m,\mu}_{r_1,r_2,\rho_1,\rho_2})$, then
$\sigma(A)-\sigma_p(A)\in S^{m-(r_1-r_2),\mu-(\rho_2-\rho_1)}_{r_1,r_2,\rho_1,\rho_2}$. Since in the sequel we will generally have $r_1>r_2$ and 
$\rho_2>\rho_1$, we recover an analog of the usual notion of principal part of a symbol in the (generalized) $\SG$-setting.

The next Lemma \ref{symcalculus} is the main composition result for the generalized parameter-dependent $\SG$ symbols introduced in the previous section. 
This result follows from the properties of the symbol calculus. Again, we omit the proof, and refer to \cite{Cord} and \cite{Yagdjian}. 

\begin{lem} \label{symcalculus}
Let $a\in \SGH{m_1}{\mu_1}{\kappa_1}{\ell_1}{N}$ and $b\in \SGH{m_2}{\mu_2}{\kappa_2}{\ell_2}{N}$ be two symbols that are constant in $Z_\pd(N)$. 
Then, the composed operator $\Op(c)=\Op(a)\, \Op(b)$ admits a symbol $c\in\SGH{m_1+m_2}{\mu_1+\mu_2}{\kappa_1+\kappa_2}{\ell_1+\ell_2}{N}$ which satisfies 
$$
c(t,x,\xi)\sim \sum\limits_{\alpha}
\textstyle{\dfrac{i^{|\alpha|}}{\alpha!}} D_\xi^\alpha a(t,x,\xi) D_x^\alpha b(t,x,\xi),
$$
modulo a regularizing symbol from $C^\infty \Big([0,T], S^{-\infty,-\infty}(\R^d\times\R^d)\Big)$.
\end{lem}

Given the composition rule we can now discuss the parametrix of an operator $\Op(a)$. We will denote the parametrix of $\Op(a)$ by $\Op(a)^{\sharp}$.
This means that the following equations hold true, with $I$ the identity operator, modulo $C^\iy\Big([0,T],\Op(S^{-\iy,-\iy})\Big)$:
$$
\Op(a)\Op(a)^\sharp -I=0 \textnormal{   and    } \Op(a)^{\sharp} \Op(a)-I=0.
$$
It is well known that, when $a$ is an elliptic $\SG$-symbol of order $(m,\mu)$, $\Op(a)$ admits a parametrix $\Op(a)^{\sharp}$ of order $(-m,-\mu)$. 
A similar statement holds true for operators associated with hypoelliptic symbols, but in this case the parametrix has a different order (see, for instance, Theorem 1.3.6 in \cite{Nicola_Rodino}). The result
extends to matrix-valued symbols, as we show explicitly in the next Lemma \ref{inv}.

\begin{lem}\label{inv}
Assume that $a$ is a matrix-valued symbol with entries in $\SGH{0}{0}{0}{0}{N}$ and that there exists $a^\prime \in \SGH{0}{0}{0}{0}{N}$,  $a^{\prime}$ constant in $Z_\pd(N)$, such that $a-a^\prime \in L_\infty \Big([0,T],\; \SGs^{-\ve, -\ve}(\R^{2d})\Big)$ for some $\ve >0$. If $\Op(a)$ is elliptic, that is, $|\det(a(t,x,\xi))|\geq C>0$ for all $(t,x,\xi)\in [0,T]\times \R^{2d}$, then, there exists a parametrix $\Op(a)^{\sharp}=\Op(a^{\sharp})$, such that $a^\sharp \in \SGH{0}{0}{0}{0}{N}$. Moreover, $a^\sharp - (a^\prime)^{-1} \in L_\infty \Big([0,T],\; \SGs^{-\ve, -\ve}(\R^{2d})\Big)$.
\end{lem}

\begin{proof}
We set $a^\sharp_0(t,x,\xi):= (a^\prime(t,x,\xi))^{-1}$. By a standard argument, it turns out that the matrix-valued symbol $a_0^\sharp$ belongs to $\SGH{0}{0}{0}{0}{N}$. Using Lemma \ref{symcalculus}, we can then define symbols $a^\sharp_k$ by means of the recursive scheme
$$
\sum^k_{|\al|=1} \textstyle{\frac{1}{\al!}} \Big(D^\al_\xi
a(t,x,\xi)\Big) \Big(\pa_x^\al
a^\sharp_{k-|\al|}(t,x,\xi)\Big)=:-a(t,x,\xi) a_k^\sharp(t,x,\xi).
$$
By the hypotheses, $a_k^\sharp(t,x,\xi)\equiv 0$ in $Z_\pd(N)$, while the calculus imply $a_k^\sharp \in \SGH{-k}{-k}{0}{0}{N}$, $k\ge1$.
Employing Lemma \ref{lem:asymptotic}, we obtain a symbol $a_R^\sharp \in \SGH{0}{0}{0}{0}{N}$ and a right parametrix $\Op(a_R^\sharp)$ such that
$$
a_R^\sharp - \sum^{k-1}_{l=0} a_l^\sharp \in \SGH{-k}{-k}{0}{0}{N} \;,\;a_R^\sharp(t,x,\xi)= a_0^\sharp(t,x,\xi) \; {\rm in}\; Z_\pd(N),
$$
and
$$
\Op(a)\Op(a^{\sharp}_R) - I \in C^\iy([0,T], \Op(S^{-\iy,-\iy})),
$$ 
where $I$ denotes the identity operator. In a completely similar fashion, the existence of a left parametrix $\Op(a^{\sharp}_L)$ such that $\Op(a^{\sharp}_L)\Op(a) - I \in C^\iy([0,T],\Op(S^{-\iy,-\iy}))$ can be shown. By a standard argument, it follows that $\Op(a^{\sharp}_L)$ and $\Op(a^{\sharp}_R)$ coincide modulo  $C^\iy([0,T], \Op(S^{-\iy,-\iy}))$. Then, we 
have proven the existence of the parametrix, defined by $$\Op(a)^\sharp=\Op(a^{\sharp}_L),$$ which is uniquely determined modulo $C^\iy([0,T], \Op(S^{-\iy,-\iy}))$. 
\end{proof}

\subsection{Fourier integral operators of $\SG$-type}
We review and extend notions about the class of $\SG$-Fourier integral operators in the global setting on $\R^d$, 
by first defining a parameter-dependent admissible class of phase functions in the $\SG$ context. 

\begin{defn}
A real-valued function $\varphi\in C^{\infty}([0,T]; S^{1,1}(\mathbb{R}^{2d}))$ is called a (smooth family of) \textit{simple phase} if 
$$
\norm{\nabla_{\xi}\varphi(t, x,\xi)}\asymp \norm{x}\quad\textnormal{and}\quad\norm{\nabla_{x}\varphi(t, x,\xi)}\asymp \norm{\xi}
$$ 
are fulfilled, uniformly with respect to $x, \,\xi\in\R^{2d}$ and $t\in [0,T]$. Moreover, the (smooth family of) 
simple phase function is called \textit{regular} if $$|\textnormal{det}(\nabla_{\xi}\nabla_{x}\varphi(t, x,\xi))|\geq c,$$ for some $c>0$ independent of $t, x$ and $\xi$. 
\end{defn}

In \cite{CoTo1}, it was discussed that the regular phase function $\varphi$ defines two globally invertible (families of) mappings, namely, $\xi\rightarrow \varphi'_x(t, x,\xi)$ and $x\rightarrow \varphi'_{\xi}(t, x,\xi)$. Then, the mappings generated by the first derivatives of the admissible regular phase functions give rise to $\SG$-diffeomorphism with $S^{0,0}$ parameter-dependence.

\begin{defn}\label{def:sgfios}
The generalized Fourier integral operator $\Op_{\varphi(t)}(a(t))$ of SG type I, with phase $\varphi$ and amplitude $a$, is a linear operator given by 
$$
[\Op_{\varphi(t)}(a(t))u](t,x)=(2\pi)^{-d}\int\int e^{i(\varphi(t, x,\xi)-y\cdot\xi)}a(t,x,\xi)u(y)dyd\xi,
$$
and the generalized FIO $\Op^{*}_{\varphi(t)}(b(t))$ of SG type II, with phase $\varphi$ and amplitude $b$, is a linear operator given by 
$$
[\Op_{\varphi(t)}^{*}(b(t))u](t,x)=(2\pi)^{-d}\int\int e^{i(x\cdot\xi-\varphi(t, y,\xi))}\overline{b(t,y,\xi)}u(y)dyd\xi.
$$
\end{defn}
Suppose $a,b$ and $\varphi$ are given as in Definition \ref{def:sgfios}. Then, the parameter-dependent operators $\Op_{\varphi(t)}(a(t))$ and $\Op_{\varphi(t)}^{*}(b(t)$ are linear and continuous on $\mathcal{S}(\mathbb{R}^{d})$, and uniquely extendable to linear and continuous operators on $\mathcal{S}'(\mathbb{R}^{d})$.

We state the next Theorem \ref{thm:composition} about composition of a Fourier integral operator with a pseudodifferential operator; for the sake of brevity we omit the proof that follows the approaches of \cite{CoTo1,Ruzhansky_Sugimoto}.

\begin{thm}\label{thm:composition}
Let $\phi\in C^{\infty}([0,T];S^{1,1})$ be a smooth family of simple and regular phase functions, and let $b\in \SGH{m_1}{\mu_1}{\kappa_1}{\ell_1}{N}$ with $\textnormal{supp}(b)\subset Z_{\hyp}(N)$. Let $p\in \SGH{m_2}{\mu_2}{\kappa_2}{\ell_2}{N}$ with $\textnormal{supp}(p)\subset Z_{\hyp}(N)$. Then, the composition 
$$
\Op_{\varphi(t)}(c(t))=\Op(p(t)) \Op_{\varphi(t)}(b(t))
$$
is a smooth family of Fourier integral operators with amplitude 
$$
c(t,x,\xi)\in \SGH{m_1+m_2}{\mu_1+\mu_2}{\kappa_1+\kappa_2}{\ell_1+\ell_2}{N},
$$
supported in $Z_{\hyp}(N)$ with a suitable choice of $N$. Moreover, we have the asymptotic expansion
$$
c(t,x,\xi)\sim \sum_{\alpha}\frac{i^{|\alpha|}}{\alpha !}(D^{\alpha}_{\xi}p)(t,x,\nabla_{x}\varphi(t,x,\xi))D^{\alpha}_{y}[e^{i\Phi(t,x,y,\xi)}b(t, y,\xi)]_{y=x},
$$
where $\Phi(t,x,y,\xi)=\phi(t,y,\xi)-\phi(t,x,\xi)+(x-y)\cdot \nabla_{x}\phi(t,x,\xi)$.
\end{thm}

We conclude the section by recalling some basic facts concerning the Sobolev-Kato spaces, as well as the boundedness properties
of the operators we treated above on such scale of spaces. We refer to \cite{Cord,Coriasco:998.1,Coriasco,Coriasco1} for 
the basic tools employed in the proofs and leave the details for the reader. 

\begin{prop}
The following properties of the weighted Sobolev spaces $H^{s,\sigma}(\mathbb{R}^{d})$ hold true.
\begin{enumerate}
\item $L^2(\mathbb{R}^{d})=H^{0,0}(\mathbb{R}^{d})$.
\item If $s_1\leq s_2$ and $\sigma_1\leq \sigma_2$, then $H^{s_2,\sigma_2}(\mathbb{R}^{d})\hookrightarrow H^{s_1,\sigma_1}(\mathbb{R}^{d})$. Moreover, if $s_1< s_2$ and 
$\sigma_1< \sigma_2$, then the embedding $H^{s_2,\sigma_2}(\mathbb{R}^{d})\hookrightarrow H^{s_1,\sigma_1}(\mathbb{R}^{d})$ is compact.  
\item $\cap_{s,\sigma\in \mathbb{R}}H^{s,\sigma}(\mathbb{R}^{d})=\scS(\R^{d}) \textnormal{  and  }\cup_{s,\sigma\in \mathbb{R}}H^{s,\sigma}(\mathbb{R}^{d})=\scS'(\R^{d})$.
\item The operator $\Op(\omega_{t,\tau})$, $\omega_{t,\tau}(x,\xi)=\norm{x}^{t}\norm{\xi}^{\tau}$, is a continous, invertible, linear operator from $H^{s,\sigma}(\mathbb{R}^{d})$ to the space $H^{s-t,\sigma-\tau}(\mathbb{R}^{d})$. In particular,  $u\in H^{s,\sigma}(\mathbb{R}^{d})$ if and only if $\Op(\omega_{s,\sigma})u\in L^{2}(\mathbb{R}^{d})$. 
\end{enumerate} 
\end{prop}

\begin{thm}\label{regular_fio}
Let the real-valued phase functions $\varphi(t)$ be simple and regular. Let $a\in \SGH{0}{0}{0}{0}{N}$. Then, the parameter-dependent 
operator $\Op_{\varphi(t))}(a(t))$ is $L^{2}$-bounded, and satisfies, for some constant $C>0$,
$$
||\Op_{\varphi(t)}(a(t))||_{L^{2}\rightarrow L^{2}}\leq C \sup_{|\alpha|,|\beta|\leq 2d+1} ||\partial^{\alpha}_{y}\partial^{\beta}_{\xi}a(t,y,\xi)||_{L^{\infty}}.
$$
Moreover, $\Op_{\varphi}(a)$ is a bounded linear operator from $H^{s,\sigma}(\mathbb{R}^{d})$ to $H^{s,\sigma}(\mathbb{R}^{d})$ for all $s,\sigma\in \mathbb{R}$.
\end{thm}
The proof of Theorem \ref{regular_fio} follows with a minor modification of the proof of \cite[Theorem 3.3]{Ruzhansky_Sugimoto}, see also \cite{Coriasco:998.1}.
\begin{rem}
	As above, we remark that the results in this section extend to operators defined by means of amplitude hierarchies associated with the subdivision into zones, and
	based on the more general classes $\SGs^{m,\mu}_{r_1,r_2,\rho_1,\rho_2}$, $r_j,\rho_j\ge0$, $j=1,2$, 
	$r_2,\rho _1<1$, and, in particular, on the classes $\SGs^{m,\mu}_{(\ve)}$, $\ve\in(0,1)$. We will tacitly make use of such variant results whenever they will be needed.
\end{rem}

\section{Diagonalization procedure}\label{sec:diag}
\setcounter{equation}{0}
%
%
%

The so-called perfect diagonalization \cite{KR,Yagdjian} is a procedure to switch from \eqref{eq:main} to a $2\times2$ system, and to suitably 
``decouple the equations'' of the system, so that each of them can be solved by means of Fourier integral operators separately. 
It consists of various steps, taking into account the subdivision into zones of the phase space. Similar to \cite{KR}, the first step is carried out in all the zones, 
the second step is carried out in the hyperbolic zone $Z_{\hyp}(N)$ (that is, away from $t=0$), and the third
step allows a refinement in the regular zone $Z_\reg(N)$. This can be achieved by means of the smoothness properties of the symbol class $\SGO{m}{\mu}{\kappa}{\ell}{N}$.  
The diagonalization produces then equations which are equivalent to \eqref{eq:main}, modulo rapidly decreasing/smoothing elements. Since we are interested in the
smoothness and decay properties of the solutions, such terms do not affect the claims, so we will often ignore/avoid writing them, and consider equalities modulo
such remainders.

Let us denote by $\rho(t,x,\xi)$ the positive root of the equation
$$
  \rho(t,x,\xi)^{2}=1+ \frac{\lambda(t)^{2}}{\Lambda(t)}\norm{x}\norm{\xi}\ln(\norm{x}\norm{\xi}).
$$
The function $\rho(t,x,\xi)$ is monotonic with respect to $t\in[0,T]$ , since $c_1>1/2$ in \eqref{eq:lambda'_control}. In the next Lemma \ref{lem:rho} further properties
of $\rho(t,x,\xi)$ are proved.
\begin{lem}
\label{lem:rho}
The function $\rho(t,x,\xi)$ has the following properties:
\begin{align*}
\bullet\;&\rho \in \SGH{1}{1}{1}{0}{N},  \text{ for some } N>0;\\
\bullet\;& \rho\in C\big([0,T], S^{\frac{1}{2}+\eps,\frac{1}{2}+\eps}\big), \text{ for every } \eps>0; \\
\bullet\;& \partial_t^j\rho\in C\big((0,T], S^{j+\eps,j+\eps}\big), \text{ for every } \eps>0 \text{ and } j\geq 1.
\end{align*}
In particular, for any $(t,x,\xi)\in Z_{\pd}(N)$, it holds
\begin{equation*}
|\partial_t^jD_x^\alpha D_\xi^\beta \rho(t,x,\xi)|\lesssim \frac{\lambda(t)}{\sqrt{\Lambda(t)}}\bigg(\frac{1}{\sqrt{\Lambda(t)}}\bigg)^j\<x\>^{\frac{1}{2}+\eps-|\alpha|}\<\xi\>^{\frac{1}{2}+\eps-|\beta|},
\end{equation*}
for any $\alpha,\beta\in \N^d$ such that $|\alpha|+|\beta|\geq 1$.
\end{lem}
\begin{proof}
The claims follow by a modification of the argument in \cite[Lemma 2.1.27]{Yagdjian}. In particular, it is straightforward to prove the inequality
\begin{equation}
\label{eq:useful_rho}
 |\partial_x^\alpha\partial_\xi^\beta \partial_t^j\rho(t,x,\xi)|\leq C_{\alpha \beta} \<x\>^{-|\alpha|}\<\xi\>^{-|\beta|}|\partial_t^j\rho(t,x,\xi)|, \quad j=0,1.
\end{equation}
Indeed, for small $t\in [0,T]$, we recall that $\frac{\lambda^2(t)}{\Lambda(t)}\in C^\infty([0,T])$. On the other hand, for $t\in (0,T]$ we can apply condition \eqref{eq:lambda'_control}.
The remaining details are left for the reader.
\end{proof}
For a given $N>0$, we define the symbol 
\begin{equation}\label{eq:hdef}
	\begin{aligned}
		h(t,x,\xi)&=\rho(t,x,\xi)\;\chi\!\left(\frac{\Lambda(t)\norm{x}\norm{\xi}}{N\ln (\norm{x}\norm{\xi})}\right)
		\\
		&+\lambda(t)\norm{x}\norm{\xi}\left[1-\chi\!\left(\frac{\Lambda(t)\norm{x}\norm{\xi}}{N\ln (\norm{x}\norm{\xi})}\right)\right],
	\end{aligned}
\end{equation}
where $\chi\in C_0^\infty(\R)$ is a cutoff function such that $\chi(\eta) \equiv 1$ for $|\eta| \leq 1$,\, $\chi(\eta) \equiv 0$ for $|\eta| \geq 2$, and $0\leq \chi(\eta) \leq 1$. 
\begin{lem}
\label{lem:h}
The parameter-dependent symbol $h$ in \eqref{eq:hdef} has the following properties:
\begin{align}
\bullet\;&\notag\partial_t^j h\in C([0,T],S^{1+j,1+j});\\
\bullet\;& \notag h(t,x,\xi)\in \SGH{1}{1}{1}{0}{N}, \text{ for some } N>0.
\end{align}
Moreover, there exist constants $c,C>0$ such that, for all $ (t,x,\xi)\in [0,T]\times \R^{2d}$, it holds
\begin{equation}
 \label{eq:h_elliptic} 
\max\{c,\lambda(t)\<x\>\<\xi\>\}\leq\,h(t,x,\xi)\leq C\<x\>\<\xi\>,
\end{equation}
and, for all $\alpha$ and $\beta\in \N^d$,
\[|\partial_x^\alpha\partial_{\xi}^\beta h(t,x,\xi)|\leq C_\alpha\<x\>^{-|\alpha|}\<\xi\>^{-|\beta|} h(t,x,\xi). \]
Finally, for all $(t,x,\xi)\in Z_\pd(2N)$ and $\alpha,\beta\in \N^d$, we have
\begin{align*}
|\partial_t\partial_x^\alpha\partial_{\xi}^\beta h(t,x,\xi)|\leq C_\alpha\<x\>^{-|\alpha|}\<\xi\>^{-|\beta|} |\partial_t\rho(t,x,\xi)| .
\end{align*}
\end{lem}
\begin{rem}
Observe that, in view of Definition \ref{def: SG-ellipticity}, taking into accout \eqref{eq:h_elliptic}, it turns out that the parameter-dependent symbol 
$h$ is globally hypoelliptic. In particular, in the hyperbolic zone it is actually an elliptic symbol.
\end{rem}
Setting $U(t)=(U_1(t),U_2(t))^{T}=(\Op(h(t))u(t), D_{t}u(t))^{T}$, the Cauchy problem \eqref{eq:main} is equivalent, modulo $C^1([0,T],\scS(\R^d))$, to the Cauchy problem
\begin{equation} \label{main_system}
\begin{cases}
	D_tU(t) - K(t)U(t)= G(t)
	\\
	U(0)=U_0,
\end{cases}
\end{equation}
where $$G(t,x)=(0, g(t,x))^{T}, \quad U_0(x)=([\Op(h(0))\phi](x), -i\psi(x))^{T},$$ 
$K(t)=A(t) - (D_tH)(t) \, H(t)^\sharp$, and the matrix-valued operators $H(t)$ and $A(t)$ are 
\begin{small}
$$
A(t):= \left(\begin{array}{cc}  0 & \Op(h(t))\\ \Op(a(t))\Op(h(t))^{\sharp} & 0 \end{array}\right)\ \textnormal{and}\ H(t):= \left(\begin{array}{cc} \Op(h(t)) & 0 \\ 0 & 1 \end{array}\right).
$$
\end{small}
As above, $P^{\sharp}$ denotes the parametrix of the operator $P$. Indeed $H(t)^\sharp$ exists, since $\det\sigma(H(t))=h(t)$, and $h(t)$ is hypoelliptic. In the next Lemma \ref{lem:parametrix}
we deduce the properties of the matrix  $K(t)$ of the coefficients of the system in \eqref{main_system}. 
\begin{lem}
\label{lem:parametrix}
The matrix-valued parameter-dependent symbol $\si(K)$  belongs to $\SGP{2N} \cap \SGH {1}{1}{1}{0}{N}$.
\end{lem}
\begin{proof}
On the one hand, as a consequence of Lemma \ref{lem:rho} and \ref{lem:h}, it follows that $D_tH\in \SGH {1}{1}{1}{0}{N}$. Then, by the calculus, 
$(D_tH)  H^{\sharp}\in \SGH {1}{1}{1}{0}{N}$, since $H^\sharp \in \SGH {0}{0}{0}{0}{N}$.
 On the other hand, the assumptions \eqref{eq:coefficients_assumptions} on the coefficients $a_j$, $b_j$, $j=1,\dots,d$, and $c$, guarantee that 
 $\Op(a)\Op(h)^\sharp\in \SGH{1}{1}{1}{0}{N} \cap \SGP{2N}$, as a consequence of \eqref{eq:h_elliptic} (see also Remark \ref{rem:a_classes}). 
 Since $H^\sharp\in \SGH{-1}{-1}{-1}{0}{N}$ and it is uniformly bounded in $Z_\pd(2N)$, 
 we conclude that $\Op(a)\Op(h)^\sharp\in \SGH{1}{1}{1}{0}{N}\cap S^\pd_{2N}$. 
\end{proof}

Let us now define   
\begin{equation*}
\begin{aligned}
\mathfrak{t}_j(t,x,\xi)&= d_j \rho(t,x,\xi) \chi\left(\frac{\Lambda(t)\norm{x}\norm{\xi}}{N\ln (\norm{x}\norm{\xi})}\right) 
\\
&+ \tau_j(t,x,\xi)\Big(1-\chi\left(\frac{\Lambda(t)\norm{x}\norm{\xi}}{N\ln (\norm{x}\norm{\xi})}\right)\Big),
\end{aligned}
\end{equation*}
where $\tau_{j}(t,x,\xi)=d_{j}\sqrt{a(t,x,\xi)}$, $j=1,2$, with $d_{2}=-d_{1}=1$, are the roots of the complete symbol \eqref{eq:complete_symbol} of $L$. The symbol 
$\mathfrak{t}_j(t,x,\xi)$, $j=1,2$, are introduced to allow the construction of the fundamental solution close to $t=0$, due to jump in the multiplicity of the characteristic
roots there. 
\begin{lem} The functions $\mathfrak{t}_1$ and $\mathfrak{t}_2$ satisfy the following properties:
\begin{enumerate}
\item[a)] $\mathfrak{t}_{k} \in \SGP{2N}\cap \SGH{1}{1}{1}{0}{N}$, $k=1,2$, with $\mathfrak{t}_2-\mathfrak{t}_1=2\mathfrak{t}_2$. In particular, for any $(t,x,\xi)\in Z_{\pd}(N)$ it holds
\begin{equation*}
	|D_x^\alpha D_\xi^\beta \mathfrak{t}_i(t,x,\xi)|\lesssim \frac{\lambda(t)}{\sqrt{\Lambda(t)}}\<x\>^{\frac{1}{2}+\eps-|\alpha|}\<\xi\>^{\frac{1}{2}+\eps-|\beta|}, \quad k=1,2,
\end{equation*}
for any $\alpha,\beta\in \N^n$ such that $|\alpha|+|\beta|\geq 1$;
\item[b)] $\partial^{j}_{t}\mathfrak{t}_{k}\in L_{\infty}([0,T], S^{1+j,1+j})$, $k=1,2$, for all $j=0,1,\dots$;
\item[c)] $\partial^{j}_{t}(\mathfrak{t}_{k}/h)\in L_{\infty}([0,T], S^{1+j,1+j})$, $k=1,2$, for all $j=0,1,\dots$;
\item[d)] $|D_x^\beta D_\xi^\alpha D_{t} \frac{\mathfrak{t}_k(t,x,\xi)}{\rho(t,x,\xi)}|\leq C_{\alpha\beta}\norm{\xi}^{-|\alpha|}\norm{x}^{-|\beta|}\frac{|\partial_{t}\rho(t,x,\xi)|}{\rho(t,x,\xi)}$,
$k=1,2$, \textnormal{ for } $(t,x,\xi)\in Z_{\pd}(2N)$.
\end{enumerate}
\end{lem}

To begin the diagonalization of the system in \eqref{main_system}, we introduce the matrix-valued parameter-dependent operator
\begin{equation}\label{eq:opM}
M= \left(\begin{array}{cc} I & I \\ \Op(\mathfrak{t}_{1}) \Op(h)^{\sharp} & \Op(\mathfrak{t}_{2}) \Op(h)^{\sharp} \end{array}\right),
\end{equation}
whose principal symbol symbol is
$$
m_0=\sigma_p(M) = \left(\begin{array}{ll} 1 & 1\\ \frac{\mathfrak{t}_1}{h} & \frac{\mathfrak{t}_2}{h}\end{array}\right).
$$
Since $\det(m_0)=\frac{\mathfrak{t}_2-\mathfrak{t}_1}{h}=\frac{2\mathfrak{t}_2}{h}\gtrsim 1$,  for $(t,x,\xi)\in [0,T]\times \mathbb{R}^{2d}$, the operator $M$ is elliptic.
By Lemma \ref{lem:parametrix}, $M$ then admits a parametrix $M^{\sharp}$. Note that, by their definition, the symbols of $M, M^{\sharp}$ are constant in $Z_{\pd}(N)$, while their entries belong to $\SGH{0}{0}{0}{0}{N}$. 

We look for the fundamental solution $E=E(t,s)$ to (\ref{main_system}), that is, an operator family $E(t,s)$ satisfying 
\begin{equation}\label{fundamental_solution}
D_tE(t,s) - K(t)E(t,s)  = 0\; , \; E(s,s) = I.
\end{equation}
In view of the properties of the calculus of the generalized parameter-dependent $\SG$-operators, established in the previous sections, we can
adapt the arguments in \cite{KR,Yagdjian}. We first set $E_0(t,s) = M^\sharp(t) E(t,s)$, which then satisfies
\begin{align*}
\nonumber D_{t}E_{0} &= M^{\sharp}(A+(D_tH)H^{\sharp})E+(D_tM^{\sharp})E\\
\nonumber &= (M^{\sharp}AM)E_{0}+(D_tM^{\sharp}+M^\sharp (D_t H)H^{\sharp})ME_{0}+\hat R_{1}E \\
\nonumber &= \mathcal{D}E_{0} + \mathcal{B}_{1}E_{0}+\hat{R}_{1}E,
\end{align*}
where
$$
\mathcal{D} = \Op(\sigma_p(M^{\sharp}AM)), \; \mathcal{B}_{1}= (M^{\sharp}AM-\mathcal{D})+(D_tM^{\sharp}+M^\sharp (D_tH)H^{\sharp})M.
$$
and $\hat{R}_{1} \in C^\iy \Big([0,T], \Op(S^{-\infty,-\infty})\Big)$ is a smooth family of regularizing operators. 

We first examine $\mathcal{D}$. By Lemmas \ref{symcalculus} and \ref{inv}, looking at the top order terms of the asymptotic expansions of the involved compositions and parametrix, we see that
\begin{equation}
\label{eq:D_symbol}
\sigma(\mathcal{D})=\sigma_p(M^{\sharp}AM)=
\begin{cases}
\left(\begin{array}{ll}\tau_{1} & 0\\ 0 & \tau_{2}\end{array}\right) &\textnormal{ in }\ Z_{\hyp}(2N),
\\
\left(\begin{array}{cc} \frac{ \mathfrak{t}_1^2+\tau_1^2 }{2\mathfrak{t}_1} &
\frac{\mathfrak{t}_2^2 -\tau_2^2}{2\mathfrak{t}_2}\\
\frac{\mathfrak{t}_1^2-\tau_1^2}{2\mathfrak{t}_1} & \frac{\mathfrak{t}_2^2+\tau_2^2}{2
\mathfrak{t}_2}
\end{array}\right) & \textnormal{ in }  Z_{\pd}(2N).\rule{0mm}{1cm}
\end{cases}
\end{equation}
Then,
$$
\sigma(\mathcal{D})=\left(\begin{array}{cc} \mathfrak{t}_1 &
0\\ 0 & \mathfrak{t}_2 \end{array}\right) + \si(Q), 
$$
where $\si(Q) \in \SGP{2N} $ and $\si(Q) \equiv 0$ in $Z_{\hyp}(2N)$. Again by the definition of $M$ and the properties of the calculus, we can write
\[\sigma(M^{\sharp}AM)=\sigma(\mathcal{D})+q_{0}+r_{0},\]
where $q_{0}\equiv 0$ in $Z_{\pd}(2N)$, $q_{0}\in \SGP{N}\cap \SGH{0}{0}{0}{0}{N}$ and $r_{0}\in C^{\infty}([0,T],$ $S^{-\infty,-\infty})$. 

We now consider $\mathcal{B}_{1}$. Using the equalities
\[
	M^\sharp M=I\Rightarrow (D_t M^\sharp)M=-M^\sharp(D_tM),
\]
which hold modulo smooth families of regularizing operators, we find
\begin{align*}
\nonumber \mathcal{B}_{1}&= (M^{\sharp}AM-\mathcal{D})+(D_tM^{\sharp}+M^\sharp (D_tH)H^{\sharp})M \\
\nonumber &=(M^{\sharp}AM-\mathcal{D})+(-M^{\sharp}D_tM+M^\sharp (D_tH) H^{\sharp}M)+\tilde{R}_{1}
\end{align*}
with $\tilde{R}_{1}\in C^{\infty}([0,T]; \Op(S^{-\infty-\infty}))$. Again by the properties of the calculus, we can write
\begin{align*}
\nonumber\sigma(-M^{\sharp}D_tM&+M^\sharp (D_tH)H^{\sharp}M) \\
\nonumber &=\sigma_p(-M^{\sharp}D_tM)+\sigma_p(M^\sharp(D_tH)H^{\sharp}M) +q_{1}+r_{1},
\end{align*}
where $q_{1}\equiv 0$ in $Z_{\pd}(N)$, $q_{1}\in \SGP{N}\cap \SGH{-1}{-1}{0}{1}{N}$, and $r_{1}\in C^{\infty}([0,T],$ $S^{-\infty,-\infty})$. 
Looking at the top terms of the asymptotic expansions, by direct calculation we have 
\[\sigma_p(M^\sharp (D_tH) H^{\sharp} M)=
\begin{pmatrix} \frac{D_{t}h}{2h} & \frac{D_{t}h}{2h}\\\ \frac{D_{t}h}{2h} & \frac{D_{t}h}{2h}\rule{0mm}{5mm} \end{pmatrix}
\] 
and 
\[
\sigma_p(M^{\sharp}D_tM)=\frac{h}{2\mathfrak{t}_{2}}\left(\begin{array}{cc} -D_{t}(\frac{\mathfrak{t}_1}{h})& -D_{t}(\frac{\mathfrak{t}_2}{h})\\D_{t}(\frac{\mathfrak{t}_1}{h})&D_{t}(\frac{\mathfrak{t}_2}{h}) \rule{0mm}{5mm}\end{array}\right).
\]
Summing up, it follows that there exist $q\in \SGP{N}\cap \SGH{0}{0}{0}{0}{N}$ and $r\in C^{\infty}([0,T],S^{-\infty,-\infty})$
such that
\[
\sigma(\mathcal{B}_{1})=b_1=
\left(\begin{array}{cc}
\frac{D_t \mathfrak{t}_2}{2\mathfrak{t}_{2}} &-\frac{D_t \mathfrak{t}_2}{2\mathfrak{t}_{2}}+\frac{D_t h}{h}  \\ \frac{D_t \mathfrak{t}_2}{2\mathfrak{t}_{2}} +\frac{D_t h}{h} & \frac{D_t \mathfrak{t}_2}{2\mathfrak{t}_{2}}\rule{0mm}{5mm}
\end{array}\right)+q+r.
\]
This shows $b_1\in \SGH{0}{0}{0}{1}{N}$. Moreover, we can estimate 
\begin{equation}\label{eq:B1_bound_pd}
| D_{x}^{\beta}D_{\xi}^{\alpha}b_{1}(t,x,\xi)|\leq C_{\alpha\beta}\norm{x}^{-|\beta|}\norm{\xi}^{-|\alpha|}\left(\rho(t,x,\xi)+\frac{\partial_{t}\rho(t,x,\xi)}{\rho(t,x,\xi)}\right),
\end{equation}
for all $(t,x,\xi)\in Z_{\pd}(2N)$. Indeed, \eqref{eq:B1_bound_pd} follows by Lemma \ref{lem:rho} and \eqref{eq:useful_rho}, noticing that 
$$ |D_x^\alpha D_\xi^\beta a(t,x,\xi)|\leq C_{\alpha\,\beta}\<x\>^{-|\alpha|}\<\xi\>^{-|\beta|}\rho^2(t,x,\xi),$$ and 
$$ |D_x^\alpha D_\xi^\beta \mathfrak{t}_j(t,x,\xi)|\leq C_{\alpha\,\beta}\<x\>^{-|\alpha|}\<\xi\>^{-|\beta|}\rho(t,x,\xi), \quad j=1,2,$$
for all $(t,x,\xi)\in Z_\pd(2N)$, as a consequence of \eqref{eq:coefficients_assumptions}.

Summarizing, we proved that the fundamental solution $E(t,s)$ satisfying \eqref{fundamental_solution} can be represented, modulo
smooth families of regularizing operators, in the form $E(t,s)=M(t)E_{0}(t,s)M^{\sharp}(s)$, 
where $M$ is the matrix-valued elliptic operator \eqref{eq:opM}, with $\sigma(M)\in \SGH{0}{0}{0}{0}{N}$, and $E_{0}=E_{0}(t,s)$ solves
\begin{equation*}
	\label{eq:sysbis}
D_{t}E_{0}-\mathcal{D}E_{0}+\mathcal{B}_{1}E_{0}+R_1 E_{0}=0, E_0(s,s)=I,
\end{equation*}
with $\mathcal{D}$ the matrix-valued diagonal pseudodifferential operator with symbol given in \eqref{eq:D_symbol},  $\mathcal{B}_{1}$ 
a matrix-valued pseudodifferential operator with symbol
$b_1\in\SGP{N}\cap \SGH{0}{0}{0}{1}{N}$, satisfying \eqref{eq:B1_bound_pd} for all $(t,x,\xi)\in Z_{\pd}(2N)$, and $R_1$  a smooth family of
regularizing operator with matrix-valued symbol in $ C^{\infty}([0,T],S^{-\infty,-\infty})$.

Our next goal is to diagonalize $\mathcal{B}_{1}$ modulo $\SGH{-1}{-1}{-1}{2}{N}$. Since we are looking for the fundamental solution modulo
regularizing terms, from now on we will no longer indicate them explicitly in the computations. 

\begin{prop}\label{diag_2}
There exists an elliptic, matrix-valued pseudodifferential operator $N_{1}$ with $\sigma(N_{1})\in \SGH{0}{0}{0}{0}{N}$, $\sigma(N_{1})\equiv I$ in $Z_{\pd}(N)$,
a diagonal matrix-valued pseudodifferential operator $\mathcal{D}_{1}$ with $\sigma(\mathcal{D}_{1})\in \SGP{2N}\cap \SGH{0}{0}{0}{1}{2N}$, 
and a matrix-valued pseudodifferential operator $\mathcal{B}_{2}$ with $\sigma(\mathcal{B}_{2})\in \SGP{2N}\cap \SGH{-1}{-1}{-1}{2}{2N}$, such that 
\begin{equation*}
(D_{t}-\mathcal{D}+\mathcal{B}_{1})N_{1}=N_{1}(D_{t}-\mathcal{D}+\mathcal{D}_{1}+\mathcal{B}_{2})
\end{equation*}
holds, modulo a smooth family of regularizing operators with symbol in $C^{\infty}([0,T],S^{-\infty,-\infty})$; in particular, $\sigma(\mathcal{B}_{2})$ satisfies estimate \eqref{eq:B1_bound_pd} in $Z_\pd(2N)$.
\end{prop}

\begin{proof} 
The proof is achieved by an argument similar to those in \cite{Coriasco,KR,Yagdjian}, relying on the properties of the parameter-dependent
calculus we are using. For the sake of completeness, we here provide some details.
Let $\mathcal{D}_1$ and $N^{(1)}$ be the pseudodifferential operators with matrix-valued symbols 
$$\sigma(\mathcal{D}_{1})(t,x,\xi)=\left(1-\chi\left(\frac{\Lambda(t)\norm{x}\norm{\xi}}{N\ln(\norm{x}\norm{\xi})}\right)\right)
\left(\begin{array}{cc} b_{1,11}(t,x,\xi) & 0 \\ 0 & b_{1,22}(t,x,\xi)\end{array}\right),
$$ 
and, respectively,
\begin{align*}
\sigma(N^{(1)})(t,x,\xi)&=\left(1-\chi\left(\frac{\Lambda(t)\norm{x}\norm{\xi}}{N\ln(\norm{x}\norm{\xi})}\right)\right)\left(\begin{array}{cc} 0 & \frac{b_{1,12}(t,x,\xi)}{\mathfrak{t}_{1}-\mathfrak{t}_{2}} \\ \frac{b_{1,21}(t,x,\xi)}{\mathfrak{t}_{2}-\mathfrak{t}_{1}} & 0\end{array}\right),
\end{align*}
where there appear the entries of the matrix-valued symbol
$$b_1=\sigma(\mathcal{B}_{1})=\left(\begin{array}{cc} b_{1,11} & b_{1,12} \\ b_{1,21} & b_{1,22}\end{array}\right).$$ 
Our considerations above show that $\sigma(\mathcal{D}_{1})\in \SGH{0}{0}{0}{1}{N}$, and also $\sigma(N^{(1)})\in \SGH{-1}{-1}{-1}{1}{N}$. 
Moreover, by their definition,
$$
\sigma(\mathcal{D}_{1})(t,x,\xi)=\sigma(N^{(1)})(t,x,\xi)=0\ \ \textnormal{ for } (t,x,\xi)\in Z_{\pd}(N).
$$  
Let $N_1=I+N^{(1)}$. Then, setting 
\begin{align*}
\mathcal{B}_{2}&=(D_{t}-\mathcal{D}+\mathcal{B}_{1})(I+N^{(1)})-(I+N^{(1)})(D_{t}-\mathcal{D}+\mathcal{D}_{1})\\
 &= D_{t} N^{(1)}+[N^{(1)},\mathcal{D}]+\mathcal{B}_1N^{(1)}-N^{(1)}\mathcal{D}_{1}+\mathcal{B}_1-\mathcal{D}_{1},
\end{align*}
it follows
\[(D_t-\mathcal D+\mathcal B_1)N_1=N_1(D_t-\mathcal D +\mathcal D_1+\mathcal B_2).\]
By a direct calculation, using Lemma \ref{symcalculus}, we find 
$$\sigma_p(\mathcal{B}_1(1-\chi)-\mathcal{D}_{1}+[N^{(1)},\mathcal{D}])\equiv 0;$$
which implies 
\[ \sigma_p(\mathcal{B}_2)=\sigma_p(\tilde{\mathcal{B}}_2+\chi \mathcal B_1 N_1),\]
where 
\[\tilde{\mathcal B}_2:=D_{t} N^{(1)}+(1-\chi)\mathcal{B}_1N^{(1)}-N^{(1)}\mathcal{D}_{1}.\]
In particular, $\sigma(\mathcal{\tilde{B}}_2)\equiv 0$ in $Z_\pd(N)$ and both the symbols
$$\sigma(D_{t}N^{(1)}) \text{ and } \sigma(\mathcal{B}N^{(1)}-N^{(1)}\mathcal{D}_{1})$$
belong to $S^\pd_{2N}\cap\SGH{-1}{-1}{-1}{2}{N}$, as a consequence of \eqref{eq:B1_bound_pd}. 

On the other hand, 
$ \sigma_p(\chi\mathcal{B}_1N_1)\equiv 0 $ in $Z_\hyp(N)$ and it belongs to $S^\pd_{2N}$ too.
That is, $\mathcal{B}_2$ satisfies the desired properties. 

Now, let us show that, for a sufficiently large $N$, the pseudodifferential operator $N_{1}$ is a elliptic,
with symbol belonging to $\SGH{0}{0}{0}{0}{N}$. By its definition, $\sigma(N_{1})\equiv 1$ in $Z_{\pd}(N)$ and $n_{1}(t,x,\xi)=\sigma(N_{1})(t,x,\xi)\in \SGH{-1}{-1}{-1}{1}{N}$.
Moreover,
\begin{align}
\nonumber |\sigma_p(N_{1})(t,x,\xi)|&\leq \frac{C}{\norm{x}\norm{\xi} \lambda(t)}\left(\frac{\lambda(t)}{\Lambda(t)}\left(\frac{1}{\ln \Lambda(t)}\right)\right)\\
\nonumber &\leq \frac{C(\ln (\norm{x}\norm{\xi}))}{\norm{x}\norm{\xi}\Lambda(t)}\\
\nonumber &\leq \frac{C}{N}\ \ \textnormal{ in }  Z_{\hyp}(N).
\end{align}
Consequently, a large $N$ yields $|\sigma_p(N_{1})(t,x,\xi)|\geq \frac{1}{2}$ in $[0,T]\times \mathbb{R}^{2d}$, using $\sigma(N_{1})\equiv I$ in $Z_{\pd}(N)$.
This gives, together with Lemma \ref{inv}, the existence of a parametrix $N_{1}^{\sharp}$ with $\sigma(N_{1}^{\sharp})\in \SGH{0}{0}{0}{0}{N}$.
\end{proof}

We conclude the section with a further steps of the diagonalization, localized in $Z_{\reg}(N)\subset Z_{\hyp}(N)$. Observe that, for a symbol $p\in \SGH{-p}{-p}{-p}{p+1}{N}$,
we can estimate 
\begin{align}
\nonumber \int_{t_{x,\xi}}^t \Big|p(\tau,x,\xi)\Big| d\tau&\leq  \int ^t_{t_{x,\xi}} \textstyle{\frac{C_p}{\norm{x}^{p}\norm{\xi}^{p}\lambda(\tau)^{p}}\Big(\frac{\lambda(\tau)}{\Lambda(\tau)} \ln \frac{1}{\Lambda(\tau)}\Big)^{p+1} d\tau }\\
\nonumber &\leq \frac{C_p (\ln \norm{x}\norm{\xi})^{(p+1)}}{(\norm{x}\norm{\xi} \Lambda(t_{x,\xi}'))^p} \\
\nonumber &= \frac{C_p}{(2N)^p} (\ln \norm{x}\norm{\xi})^{(p+1)-2p}
\end{align}
where $t_{x,\xi}$ is defined as in (\ref{txxi_equation}). 
In the oscillations subzone, corresponding to $\mathcal{D}+\mathcal{B}_1$, we achieved remainder of the type $\SGH{0}{0}{0}{0}{2N} + \SGH{-1}{-1}{-1}{2}{2N}$. We can
actually improve the diagonalization scheme modulo operators with symbols from 
\begin{equation}\label{eq:defHGN}
\begin{aligned}
\mathcal{HG}_{N}&=\SGP{2N}
\\
& \cap \Big(\SGH{0}{0}{0}{0}{2N} + \SGH{-1}{-1}{-1}{2}{2N}\Big) 
\\
&\cap\Big(\bigcap\limits_{p\geq 0} \SGO{-p}{-p}{-p}{p+1}{N}\Big).
\end{aligned}
\end{equation}
\begin{thm}\label{diag_3}
There exist matrix-valued pseudodifferential operators $N_{2}$, $\mathcal{D}_{2}$, $\mathcal{B}_{\infty}$ such that 
\begin{equation*}
(D_{t}-\mathcal{D}+\mathcal{D}_{1}+\mathcal{B}_{2})N_{2}=N_{2}(D_{t}-\mathcal{D}+\mathcal{D}_{2}+\mathcal{B}_{\infty}) 
\end{equation*}
holds modulo $C^{\infty}([0,T],\Op(S^{-\infty,-\infty}))$, with the elliptic symbol $\sigma(N_{2})\in \SGH{0}{0}{0}{0}{N}$ satisfying $\sigma(N_{2})=1$ in $Z_{\pd}(N)\cup Z_\osc(N)$, the diagonal matrix $\sigma(\mathcal{D}_{2})\in \SGO{0}{0}{0}{0}{N}+\SGO{-1}{-1}{-1}{2}{N}$ vanishing in $Z_{\pd}(N)\cup Z_{\osc}(N)$ and $\sigma(\mathcal{B}_{\infty})(t,x,\xi)\in \mathcal{HG}_{N}$; in particular, $\sigma(\mathcal{B}_{\infty})$ satisfies estimate \eqref{eq:B1_bound_pd} in $Z_\pd(2N)$.
\end{thm}
We omit the details of the proof, since it follows similar lines to that of Theorem \ref{diag_2}.

\section{The parametrix of the diagonalized problem}\label{sec:param}
\setcounter{equation}{0}
%
%
%

As a consequence of Proposition \ref{diag_2} and Theorem \ref{diag_3},  in order to construct a fundamental solution $E=E(t,s)$ solving (\ref{fundamental_solution}) it is sufficient to study the system  
\begin{equation*}
D_tE - {\mathcal D}E + \mathcal{D}_{2}E + \mathcal{B}_{\infty}E + R_\infty E=0,\quad E(s,s)=I\;.
\end{equation*}
Since it is enough, for our aims, to obtain $E$ modulo smooth families of regularizing operators, we will ignore the term $R_\infty E$.
As customary in this approach, we construct the parametrix in two steps, corresponding to the diagonal terms ($D_{t}-\mathcal{D}+\mathcal{D}_{2}$) and to the non-diagonal term 
$\mathcal{B}_{\infty}$. 

Firstly, we construct the parametrix $E_{1}$ of  
$$
D_{t}E_{1}-\mathcal{D}E_{1}+\mathcal{D}_{2}E_{1}\sim 0,\ E_{1}(s,s)\sim I.
$$  
As it turns out, the parametrix $E_{1}$ is a diagonal Fourier Integral Operator. The main difficulty here is in  determining the inhomogeneous phase function of SG type. Further, to determine the amplitude of $E_{1}$, we recognize that the terms of the expression $-\mathcal{D}+\mathcal{D}_{2}$ have symbols in different classes, i.e.,  
$$
\sigma(\mathcal{D})\in \SGs_N^\pd\cap\SGH{1}{1}{1}{0}{N}
$$
and
$$
\sigma(\mathcal{D}_{2})\in \SGs_N^\pd\cap\left(\SGO{0}{0}{0}{0}{N}+\SGO{-1}{-1}{-1}{2}{N}\right).
$$
In view of this, we determine the phase function based on the top-order term $-\mathcal{D}$ and determine the amplitude using the whole term $-\mathcal{D}+\mathcal{D}_{2}$.

Secondly, we construct the parametrix $Q(t,s)$ corresponding to the non-diagonal part such that $E=E_{1}Q$. As we do not have a diagonal structure in this case, we cannot obtain the parametrix as a diagonal Fourier integral operator. Nevertheless, we can take advantage of the fact that $\si(\mathcal{B}_{\infty}) \in \mathcal{HG}_{N}$ and observe that the parametrix is indeed a pseudodifferential operator. 

\subsection{Construction of the phase functions}\label{subs:HamFlow}
Let us denote by $\vartheta = \vartheta(t,x,\xi)$ the real part of one of the functions $\mathfrak{t}_k,\; k=1,2$. We consider the Hamiltonian flow
$(q,p) = (q,p) (t,s,y,\eta) =  \psi_{s,t}(y,\eta)$, defined as the solution to 
\begin{equation}
\label{eq:Hamiltonian} 
\begin{aligned}
&& \frac{dq}{dt} = \nabla_\xi \vartheta(t,q,p),\;\;
q(s,s,y,\eta)=y,\\ && \frac{dp}{dt} = -\nabla_x \vartheta(t,q,p),\;\; p(s,s,y,\eta)=\eta.
\end{aligned}
\end{equation}
For convenience, in the sequel we will sometimes denote $q(t,s,y,\eta)$ by $q(t,s)$ and, similarly, $p(t,s,y,\eta)$  by $p(t,s)$.  Following the approach used in \cite{Yagdjian},
we prove the next result.

\begin{lem}\label{hamiltonian_flow1}
There  exists $T_{0}\in (0,T]$ such that the solution $(q(t,s,y,\eta),$ $p(t,s,y,\eta))$ to \eqref{eq:Hamiltonian} exists uniquely on $[0,T_0]^2\times \R^{d}_{y}\times \R^d_\eta$.
Moreover,
\[ \partial_t^k\partial_s^l\partial_y^\alpha\partial_{\eta}^\beta (q(t,s,y,\eta), p(t,s,y,\eta))\in C([0,T_0]^2\times \R^d_y\times \R^d_\eta).\]  
In particular, it holds
\begin{equation}
\label{eq:p,q_representation}
\begin{aligned}
p(t,s)&=\eta+\int_{s}^{t}\nabla_{x}\vartheta(\tau,q(\tau,s),p(\tau,s))d\tau, \\ q(t,s)&=y-\int_{s}^{t}\nabla_{\xi}\vartheta(\tau,q(\tau,s),p(\tau,s))d\tau,
\end{aligned}
\end{equation}
for all $(s,t)\in [0,T_0]^2$ and $(y,\eta)\in \R^{2d}$.
\end{lem}
\begin{proof}
Let us define 
\[f(t,q,p)=(\nabla_{\xi}\theta(t,q,p),-\nabla_{x}\theta(t,q,p)).\]
Since $\theta\in C([0,T],S^{1,1})\subset C ([0,T],C^\infty(\R^{2d})),$ we get that $f$ is continuous together with its partial derivatives $\partial_{p_i}f$ and $\partial_{q_i}f$ in $[0,T]\times \R^{2d}$. As a consequence (see \cite[Theorem 15]{Pontriagin}), for all $s\geq 0$ and $(y,\eta)\in \R^{2d}$ there exists $r_s>0$ and $\sigma_s>0$, depending on $s, y$ and $\eta$,
such that the solution $(p(t,\tau,\tilde{y},\tilde{\eta}),q(t,\tau,\tilde{y},\tilde{\eta}))$ to \eqref{eq:Hamiltonian} with initial condition $\tilde{y}$ and $\tilde{\eta}$ is defined and continuous with respect to the variables $t,\tau, \tilde{y},\tilde{\eta}$ for 
\[|t-\tau|<r_s, \quad |\tau-s|<\sigma_s, \quad \|(\tilde{y},\tilde{\eta})-(y,\eta)\|<\sigma.
\]
The existence of a common interval of definition $[0,T_0]$, independent of $y$ and $\eta$, is not trivial. 

Let $(y,\eta)\in \R^{2d}$, with $|y|>M$ and $|\eta|>L$. We denote 
\begin{equation*}
\begin{aligned}
&z:=(z_1, z_2):= (p/|y|, q/|\eta|);\\
&f:=(f_1(t,z_1,z_2), f_2(t,z_1,z_2))\\& \hspace{6pt}:=(|y|^{-1}\nabla_\xi \theta(t, |y|z_1, |\eta|z_2), -|\eta|^{-1}\nabla_{x}\theta(t,|y|z_1,|\eta|z_2)).
\end{aligned}
\end{equation*}
Then, $z$ satisfies the equation 
\begin{equation}
\label{eq:Hamiltonian_existence}
\frac{dz}{dt}=f(t,z), \quad z|_{t=s}=(\omega,\zeta)\in \R^{2d}, \; |\omega|=|\zeta|=1.
\end{equation}
It remains to prove that for all $|y|>M$ and $|\eta|>L$ there exists a common domain of existence for $z(t,s,\omega, \zeta, y,\eta)$. We consider 
\begin{align*}
\Pi&:=\{(t,z)\in [0,T]\times \R^{2d}:\; |z_1-\omega|+|z_2-\zeta|\leq a^2\leq 1/4\},\\
\Pi_r&:=\{(s,t,z)\in [0,T]^2\times \R^{2d}:\; t+s\leq r, (t,z)\in \Pi \}.
\end{align*}
It is obvious that there exist $M_1$ and $K$ positive constants such that
\[ |f|\leq M_1, \quad \Big|\frac{\partial f_i}{\partial z_j}\Big|\leq K, \quad i,j=1,2,\]
uniformly in $\Pi$ with respect to $|y|$ and $|\eta|$. As a consequence (see \cite{Pontriagin}, Theorem 15 and formulae (19), (22) in Section 21), the solution to \eqref{eq:Hamiltonian_existence} exists and is continuous in $\Pi_r$ if $(s,t)\in [0,T_0]^2$ with
\[T_0\leq r\leq \frac{a}{M_1}, \quad r\leq \frac{k}{4 d^2 K}, \text{ for some } k<1.\]
In particular, this solution depends smoothly on the parameters $y$ and $\eta$.
\end{proof}
Let us describe the behavior of the solution $(q(t,s,y,\eta),p(t,s,y,\eta))$. To this end we introduce an auxiliary point $\tilde t_{x,\xi}$ such that 
\[ \Lambda(\tilde t_{x,\xi})\<x\>\<\xi\>=N_1 \ln(\<x\>\<\xi\>),\]
where $N_1<N$. Clearly, it holds $\tilde t_{x,\xi}< t_{x,\xi}$.
\begin{lem}\label{lem:hamiltonian_path}
For $(s,y,\eta)\in Z_{\hyp}(N)$ there exists $N_{1}<N$ such that  $(t,q(t,s),p(t,s))\in Z_{\hyp}(N_{1})$ for all $t\in [s,T_0]$, taking $T_0>s$ sufficiently small.
\end{lem}
\begin{proof}
To describe the behavior of the solution $(q(t,s,y,\eta),p(t,s,y,\eta))$ with respect to the zones we consider the properties of the function $\vartheta(t,x,\xi)$: 
for a sufficiently small $T_{0}$,
$$
|\nabla_{x}\vartheta(t,x,\xi)|\leq c\lambda(t)\norm{\xi}  \textnormal{  and  } |\nabla_{\xi}\vartheta(t,x,\xi)|\leq c\lambda(t)\norm{x},
$$ 
for all $(t,x,\xi)\in Z_\hyp(N)$. So, we obtain, for $0\leq s\leq t\leq T_{0}$,
$$
|\frac{d}{dt}\norm{p(t)}^{2}|\leq c\norm{p(t)}^{2}\lambda(t)\textnormal{  and   } |\frac{d}{dt}\norm{q(t)}^{2}|\leq c\norm{q(t)}^{2}\lambda(t).
$$
By applying Gronwall's lemma we get
\begin{equation*}
\<p(s)\>e^{2c(\Lambda(s)-\Lambda(t))} \leq \<p(t)\>\leq \<p(s)\>e^{2c(\Lambda(t)-\Lambda(s))}
\end{equation*}
and
$$
\<q(s)\>e^{2c(\Lambda(s)-\Lambda(t))} \leq \<q(t)\>\leq \<q(s)\>e^{2c(\Lambda(t)-\Lambda(s))}.
$$

Note that the function $\Lambda(t)$ is a positive, strictly increasing, and continuous function. So, we can choose the interval $[0,T_{0}]$ sufficiently small, $t>s$ and $N_1<N$ appropriately so that we have
\begin{align}
\nonumber \frac{\Lambda(t)\norm{q(t)}\norm{p(t)}}{N_{1}\ln (\norm{q(t)}\norm{p(t)})}&\geq \frac{\Lambda(s)\norm{q(s)}\norm{p(s)}}{N\ln (\norm{q(s)}\norm{p(s)})}\frac{Ne^{c[\Lambda(s)-\Lambda(t)]}}{N_{1}(1+\frac{c[\Lambda(t)-\Lambda(s)]}{\ln (\norm{q(s)}\norm{p(s)})})}\\
\nonumber &\geq \frac{\Lambda(s)\norm{q(s)}\norm{p(s)}}{N\ln (\norm{q(s)}\norm{p(s)})}\frac{N}{N_1}\frac{e^{-c[\Lambda(t)-\Lambda(s)]}}{(1+c[\Lambda(t)-\Lambda(s)])}. 
\end{align}

Choosing $N_1<N$, there exists $\varepsilon>0$ such that $Ne^{-\omega}>N_1 (1+\omega)$ for all $\omega\in [0,\eps)$. Then, there exists $T_0$ sufficiently small such that,
for all $0\leq s\leq t\leq T_0$, it holds 
\begin{align}
\nonumber \frac{\Lambda(t)\norm{q(t)}\norm{p(t)}}{N_{1}\ln (\norm{q(t)}\norm{p(t)})}
\nonumber &\geq \frac{\Lambda(s)\norm{q(s)}\norm{p(s)}}{N\ln (\norm{q(s)}\norm{p(s)})}.
\end{align}
The desired result follows by the definition of $Z_{\hyp}(N)$.
\end{proof}
\begin{lem}
\label{lem:estimates_p,q}
Let $\alpha, \beta \in \N^d$ and $j, k\in \N$ with $j+k\in \{0,1\}$. For $T>0$ sufficiently small, there exist constants $C_{jk\alpha\beta},$ such that, 
for all $(y,\eta)\in \R^{2d}$, if $0\leq s,t \leq t_{y,\eta}\leq T$ it holds
\begin{equation}
\label{eq:p,q_estimate_s}
\begin{aligned}
&|D_t^jD_s^kD_y^\alpha D_\eta^\beta (q(t,s,y,\eta)-y)|\leq  C_{jk\alpha\beta}\<y\>^{\frac{1}{2}+\eps-|\alpha|}\<\eta\>^{-\frac{1}{2}+\eps-|\beta|}\\& \hspace{30pt}\times |\sqrt{\Lambda(t)}-\sqrt{\Lambda(s)}|^{1-j-k}\bigg(\frac{\lambda(t)}{\sqrt{\Lambda(t)}}\bigg)^j\bigg|\ln\bigg(\frac{\Lambda(t)}{\Lambda(s)}\bigg)\bigg|^k, \\
&
|D_t^jD_s^kD_y^\alpha D_\eta^\beta (p(t,s,y,\eta)-\eta)\leq C_{jk\alpha\beta}\<y\>^{-\frac{1}{2}+\eps-|\alpha|}\<\eta\>^{\frac{1}{2}+\eps-|\beta|}\\& \hspace{30pt}\times |\sqrt{\Lambda(t)}-\sqrt{\Lambda(s)}|^{1-j-k}\bigg(\frac{\lambda(t)}{\sqrt{\Lambda(t)}}\bigg)^j\bigg|\ln\bigg(\frac{\Lambda(t)}{\Lambda(s)}\bigg)\bigg|^k.
\end{aligned}
\end{equation}
Further, there exist constants $T_0, M_1$, $T_0\in(0,T]$, $M_1>M$ such that for any $j,k$ positive integers and $\alpha,\beta$ multi-indices, there exist constants $C_{jk\alpha\beta},$ such that for all $(y,\eta)\in \R^{2d}$, $|y|+|\eta|\geq M$, we have the following estimates:
\begin{itemize}
\item for $\tilde t_{y,\eta}\leq s\leq t \leq T_0$ or $0\leq s\leq \tilde t_{y,\eta}\leq t_{y,\eta}\leq t \leq T_0$ it holds 
\begin{equation}
\label{eq:p,q_estimate_hyp}
\begin{aligned}
&|D_t^jD_s^kD_y^\alpha D_\eta^\beta (q(t,s,y,\eta)-y)|\\& \hspace{20pt}\leq C_{jk\alpha\beta}\Lambda(t)\<y\>^{1-|\alpha|}\<\eta\>^{-|\beta|}\bigg(\frac{\lambda(t)}{\Lambda(t)}\ln\bigg(\frac{1}{\Lambda(t)}\bigg)\bigg)^j\bigg(\frac{\lambda(s)}{\Lambda(s)}\ln\bigg(\frac{1}{\Lambda(s)}\bigg)\bigg)^k, \\
&
|D_t^jD_s^k D_y^\alpha D_\eta^\beta  (p(t,s,y,\eta)-\eta)|\\& \hspace{20pt}\leq C_{jk\alpha\beta}\Lambda(t)\<y\>^{-|\alpha|}\<\eta\>^{1-|\beta|}\bigg(\frac{\lambda(t)}{\Lambda(t)}\ln\bigg(\frac{1}{\Lambda(t)}\bigg)\bigg)^j\bigg(\frac{\lambda(s)}{\Lambda(s)}\ln\bigg(\frac{1}{\Lambda(s)}\bigg)\bigg)^k.
\end{aligned}
\end{equation}
\end{itemize}
\end{lem}
\begin{proof}
We first note that for $(\tau,x,\xi)\in Z_\pd(N)$ it holds $\theta(\tau,x,\xi)=\rho(\tau,x,\xi)\in C([0,T],S^{\frac{1}{2}+\eps,\frac{1}{2}+\eps})$. Thus, it holds
\[ |\nabla_{x}\theta(\tau,x,\xi)|\lesssim \frac{\lambda(t)}{\sqrt{\Lambda(t)}} \<x\>^{-\frac{1}{2}+\eps}\<\xi\>^{\frac{1}{2}+\eps},\]
and
\[ |\nabla_{\xi}\theta(\tau,x,\xi)|\lesssim \frac{\lambda(t)}{\sqrt{\Lambda(t)}}\<x\>^{\frac{1}{2}+\eps}\<\xi\>^{-\frac{1}{2}+\eps}.\]
Since $(p,q)$ satisfies \eqref{eq:Hamiltonian} and $0\leq s,t,\leq t_{x,\xi}$, we find 
\begin{align*} 
\left|\frac{d}{dt}\<p(t,s)\>^2\right|&=2 \<p(t,s)\>|\nabla_{x}\theta(t,q,p)|\\&\leq 2 \frac{\lambda(t)}{\sqrt{\Lambda(t)}} \<p(t,s)\>\<q(t,s)\>^{-\frac{1}{2}+\eps}\<p(t,s)\>^{\frac{1}{2}+\eps} \\ & \leq C\frac{\lambda(t)}{\sqrt{\Lambda(t)}}\<p(t,s)\>^2.
\end{align*}
Thus, by Gronwall' inequality we may estimate for all $s,t\in [0,T_0]$,
\begin{align*}
\notag \<\eta\>\lesssim \<\eta\> e^{C(\sqrt{\Lambda(s)}-\sqrt{\Lambda(t)})}&\lesssim
\<p(t,s,y,\eta)\> \\
& \lesssim \<\eta\> e^{C(\sqrt{\Lambda(t)}-\sqrt{\Lambda(s)})}\lesssim \<\eta\>,
\end{align*}
that is, $\<p\>\sim \<\eta\>$. Similarly, we find
\[ \<q(t,s,y,\eta)\>\sim \<y\>.\]
By using representation \eqref{eq:p,q_representation}, we get also 
\begin{align*}
| p(t,s,y,\eta)-\eta|&\lesssim\<y\>^{-\frac{1}{2}+\eps}\<\eta\>^{\frac{1}{2}+\eps}|\sqrt{\Lambda(t)}-\sqrt{\Lambda(s)}|,
\\
|q(t,s,y,\eta)-y|&\lesssim \<y\>^{\frac{1}{2}+\eps}\<\eta\>^{-\frac{1}{2}+\eps}|\sqrt{\Lambda(t)}-\sqrt{\Lambda(s)}|. 
\end{align*}
In order to give the estimate for $j=k=0$ and $|\alpha|=|\beta|=1$, we define
\begin{align*}
Q_1=\nabla_{y}q(t,s,y,\eta),&\quad Q_2=\nabla_{\eta}q(t,s,y,\eta), \\ P_1=\nabla_{y}p(t,s,y,\eta),&\quad P_2=\nabla_{\eta}p(t,s,y,\eta),
\end{align*}
which satisfy
\[ \begin{cases} \frac{d}{dt} \begin{pmatrix}
Q_1 & Q_2 \\ P_1 & P_2 
\end{pmatrix}=\begin{pmatrix}
-\nabla_{\eta}\nabla_{y}\theta& -\nabla_{\eta}\nabla_{\eta}\theta \\ \nabla_{y}\nabla_{y}\theta & \nabla_{\eta}\nabla_{y}\theta 
\end{pmatrix} \begin{pmatrix}
Q_1 & Q_2 \\ P_1 & P_2 
\end{pmatrix}\\
\begin{pmatrix}
Q_1 & Q_2 \\ P_1 & P_2 
\end{pmatrix}_{t=s}= \begin{pmatrix}
I & 0 \\ 0 & I 
\end{pmatrix}.
\end{cases}\]
Let 
\begin{equation}
\label{eq:energy}
\begin{aligned}
E(t)=\|Q_1(t)-I\|^2&+\|\<y\>^{-1}\<\eta\>Q_2(t)\|^2 \\ &+\|\<y\>\<\eta\>^{-1}P_1(t)\|^2+\|P_2(t)-I\|^2.
\end{aligned}
\end{equation}
It is easy to check that the following estimates hold:
\[ \frac{d}{dt} E(t)\lesssim \frac{\lambda(t)}{\sqrt{\Lambda(t)}}\<y\>^{-\frac{1}{2}+\eps}\<\eta\>^{-\frac{1}{2}+\eps}(E(t)+\sqrt{E(t)}).\]
Then, taking account that $E(s)=0$, we conclude that
\[ E(t)\lesssim (\Lambda(t)-\Lambda(s)) \<y\>^{-1+2\eps}\<\eta\>^{-1+2\eps}.\]
This allows to conclude the desired estimates
\begin{align*}
|\nabla_y ( q(t,s,y,\eta)-y)|&\lesssim\<y\>^{-\frac{1}{2}+\eps}\<\eta\>^{-\frac{1}{2}+\eps}(\sqrt{\Lambda(t)}-\sqrt{\Lambda(s)}), \\
|\nabla_\eta (q(t,s,y,\eta)-y)|&\lesssim \<y\>^{\frac{1}{2}+\eps}\<\eta\>^{-\frac{3}{2}+\eps}(\sqrt{\Lambda(t)}-\sqrt{\Lambda(s)}),\\
|\nabla_y (p(t,s,y,\eta)-\eta)|&\lesssim \<y\>^{-\frac{3}{2}+\eps}\<\eta\>^{\frac{1}{2}+\eps}(\sqrt{\Lambda(t)}-\sqrt{\Lambda(s)}),\\
|\nabla_\eta ( p(t,s,y,\eta)-\eta)|&\lesssim\<y\>^{-\frac{1}{2}+\eps}\<\eta\>^{-\frac{1}{2}+\eps}(\sqrt{\Lambda(t)}-\sqrt{\Lambda(s)}).
\end{align*}
%
By induction on $|\alpha+\beta|$, \eqref{eq:p,q_estimate_s} can be proved for all $\alpha$ and $\beta$ when $j=k=0$.

In order to give the estimate for the derivatives with respect to $s$, we consider the auxiliary system 
\begin{equation} 
\label{eq:auxiliary_Hamilton}
\frac{dQ}{dt}=-\nabla_{\xi}\theta(t+s, Q, P), \qquad \frac{dP}{dt}=\nabla_{x}\theta(t+s, Q, P),
\end{equation}
with initial conditions
$$ Q(0,s,y,\eta)=y, \quad P(0,s,y,\eta)=\eta. $$
Then, 
\[ q(t,s,y,\eta)=Q(t-s,s,y,\eta), \qquad p(t,s,y,\eta)=P(t-s,s,y,\eta).\]
Differentiating \eqref{eq:auxiliary_Hamilton} with respect to $s$ we obtain
\begin{align*}
\frac{d}{dt}\left|\left(\frac{dQ}{ds}\right)(t-s)\right|^2&= 2 \left|\left(\frac{dQ}{ds}\right)(t-s)\right| \Big(\frac{d}{ds}(-\nabla_\xi\theta(t+s,Q,P))\Big)(t-s) \\
& = - 2 \left|\left(\frac{dQ}{ds}\right)(t-s)\right| \left(\frac{d}{dt}\nabla_\xi\theta\right)(t,Q(t-s),P(t-s))\\
&
\lesssim \left|\left(\frac{dQ}{ds}\right)(t-s)\right| \frac{\lambda(t)}{\Lambda(t)}\<q\>^{\frac{1}{2}+\eps}\<p\>^{-\frac{1}{2}+\eps},
\end{align*}
since $\partial_t \rho\in C([0,T], S^{1+\eps,1+\eps})$. Applying Gronwall' inequality, we obtain the desired estimate
\[ |\partial_s q(t,s,y, \eta)|\lesssim \<y\>^{\frac{1}{2}+\eps}\<\eta\>^{-\frac{1}{2}+\eps} \left|\ln\bigg(\frac{\Lambda(t)}{\Lambda(s)}\bigg)\right|.\]
Similarly, we may estimate
\[ |\partial_s p(t,s,y, \eta)|\lesssim \<y\>^{-\frac{1}{2}+\eps}\<\eta\>^{\frac{1}{2}+\eps} \left|\ln\bigg(\frac{\Lambda(t)}{\Lambda(s)}\bigg)\right|.\]

Finally, for $k=0$ and $j=1$, the desired estimates can be derived directly by equation \eqref{eq:Hamiltonian}. Indeed, we get
\[ |\partial_t q(t,s,y,\eta)| =|\nabla_{\xi}\theta(t,q,p)|\lesssim \frac{\lambda(t)}{\sqrt{\Lambda(t)}}\<y\>^{\frac{1}{2}+\eps}\<\eta\>^{-\frac{1}{2}+\eps}. \]
Similarly, we get the estimate for the derivatives of $p$ with respect to $t$.

By using the same approach, we can prove estimate \eqref{eq:p,q_estimate_hyp} for all $\alpha, \, \beta\in\N^d$ and $j,k\in \N$ such that $j+k\in \{0,1\}$. In particular, we note that, if 
$(s,y,\eta)\in Z_\hyp(N_1)$, then there exists $N_0<N_1$ such that $(t,q(t,s,y,\eta), p(t,s,y,\eta))$ belongs to $ Z_\hyp(N_0)$, as a consequence of Lemma \ref{lem:hamiltonian_path}. Moreover, it holds $\<p(t,s,y,\eta)\>\sim \<\eta\>$ and $\<q(t,s,y,\eta)\>\sim \<y\>$. In particular, it is easy to prove that the energy $E(t)$ defined by \eqref{eq:energy} satisfies the estimate
\[ \frac{d}{dt}E(t)\lesssim \lambda(t)(E(t)+\sqrt{E(t)}).\] 
Finally, the proof of \eqref{eq:p,q_estimate_hyp} for $k+j=1$ exploits the non-increasing monotonicity of the function $\lambda(t)/\Lambda(t)$, that is guaranteed by assumption \eqref{eq:lambda'_control}, being $C_1<1$.

In order to prove estimate \eqref{eq:p,q_estimate_hyp} for $0\leq s\leq \tilde t_{y,\eta}\leq t_{y,\eta}\leq t \leq T_0$ we combine the estimates obtained in the hyperbolic and 
pseudodifferential zones. In particular, we note that for all $\tau\in [s,\tilde{t}_{y,\eta}]$ it holds $(\tau, p(\tau,s,y,\eta),q(\tau,s,y,\eta))\in Z_\pd(N_1)$ and 
$\theta(\tau,q,p)=\rho(\tau,q,p)$. Moreover, we may estimate
\[|( D_y^\alpha D_\eta^\beta \rho)(t,p,q)|\lesssim \frac{\lambda(t)}{\sqrt{\Lambda(t)}} \<q\>^{\frac{1}{2}-|\alpha|}\<p\>^{\frac{1}{2}-|\beta|}\sqrt{\ln(\<q\>\<p\>)},\]
and then
\begin{align*} 
\int_s^{\tilde{t}_{y,\eta}} |D_y^\alpha D_\eta^\beta \theta(\tau,q,p)|\, d\tau& \lesssim \sqrt{\Lambda(\tilde{t}_{y,\eta})} \<q\>^{\frac{1}{2}-|\alpha|}\<p\>^{\frac{1}{2}-|\beta|}\sqrt{\ln(\<q\>\<p\>)}\\
&= \frac{1}{N} \Lambda(t)\<q\>^{1-|\alpha|}\<p\>^{1-|\beta|},
\end{align*}
being 
\[\ln(\<q\>\<p\>)\lesssim \Lambda(t)\<q\>\<p\>,\]
for all $(t,q,p)\in Z_\hyp(N)$.
\end{proof}
\begin{rem}
As a conquence of Lemma \ref{lem:estimates_p,q} we may conclude that $q(t,s)$ and $p(t,s)$ satisfy 
\begin{eqnarray*}
	\label{q_conditions} \textstyle{\frac{q(t,s)-y}{\Lambda(t)-\Lambda(s)}}\;,\;
	\pa_t q(t,s),\; \pa_s q(t,s) \in L_\iy\Big([0,T_0]^2,S^{1,0}(\R^n_y \times \R^n_\eta)\Big),\\ 
	\label{p_conditions} \textstyle{\frac{p(t,s)-\eta}{\Lambda(t)-\Lambda(s)}}\;,\; \pa_t p(t,s),\; \pa_s p(t,s) \in L_\iy\Big([0,T_0]^2,S^{0,1}(\R^n_y \times \R^n_\eta)\Big). 
	\end{eqnarray*}
\end{rem}
The obtained information about the behavior of the Hamiltonian flow allow to prove the following result.
\begin{lem}
\label{lem:hamiltonian_inverse}
There exists a constant $T_1$ with $0<T_1\leq T_0$ such that the mappings 
\begin{align*} 
x&=q(t,s,\cdot,\eta): y\in \R^d \to x\in \R^d, \\
\xi&= p(t,s,y, \cdot): \eta\in \R^d \to \xi\in \R^d,
\end{align*}
with parameters $(t,s,\eta)$, $s,t\in [0,T_1]$, both admit an inverse mapping $y(t,s,x,\eta)$ and $\eta(t,s,x,\xi)$, respectively, 
satisfying the following estimates, for $j=0,\, 1$ and $\alpha,\beta\in\N^d$, with suitable positive constants $C_{j\alpha\beta}$:
\begin{itemize}
\item for $0\leq s,t \leq t_{x,\xi}$
\begin{equation*}
\begin{aligned}
&
|D_t^jD_x^\alpha D_\xi^\beta (y(t,s,x,\xi)-x)|\leq  C_{j\alpha\beta}\<x\>^{\frac{1}{2}+\eps-|\alpha|}\<\xi\>^{-\frac{1}{2}+\eps-|\beta|}\\& \hspace{80pt} \times|\sqrt{\Lambda(t)}-\sqrt{\Lambda(s)}|^{1-j}\bigg(\frac{\lambda(t)}{\sqrt{\Lambda(t)}}\bigg)^j
\end{aligned}
\end{equation*}
and 
\begin{equation*}
	\begin{aligned}
		&|D_t^jD_x^\alpha D_\xi^\beta (\eta(t,s,x,\xi)-\xi)|\leq  C_{j\alpha\beta}\<x\>^{-\frac{1}{2}+\eps-|\alpha|}\<\xi\>^{\frac{1}{2}+\eps-|\beta|}\\& \hspace{80pt} |\sqrt{\Lambda(t)}-\sqrt{\Lambda(s)}|^{1-j}\bigg(\frac{\lambda(t)}{\sqrt{\Lambda(t)}}\bigg)^j;
	\end{aligned}
\end{equation*}
\item for $\tilde t_{x,\xi}\leq s\leq t \leq T_0$ or $0\leq s\leq \tilde t_{x,\xi}\leq t_{x,\xi}\leq t \leq T_0$  
\begin{equation*}
\begin{aligned}
&|D_t^jD_x^\alpha D_\xi^\beta (y(t,s,x,\xi)-x)|\\& \hspace{50pt}\leq C_{j\alpha\beta}\Lambda(t)\<x\>^{1-|\alpha|}\<\xi\>^{-|\beta|}\bigg(\frac{\lambda(t)}{\Lambda(t)}\ln\bigg(\frac{1}{\Lambda(t)}\bigg)\bigg)^j
\end{aligned}
\end{equation*}
and 
\begin{equation*}
	\begin{aligned}
		&|D_t^jD_x^\alpha D_\xi^\beta (\eta(t,s,x,\xi)-\xi)|\\& \hspace{50pt}\leq C_{j\alpha\beta}\Lambda(t)\<x\>^{-|\alpha|}\<\xi\>^{1-|\beta|}\bigg(\frac{\lambda(t)}{\Lambda(t)}\ln\bigg(\frac{1}{\Lambda(t)}\bigg)\bigg)^j.
	\end{aligned}
\end{equation*}
\end{itemize}
\end{lem}
\begin{proof}
Applying Lemma \ref{lem:hamiltonian_path}, it is easy to get that there exists $T_1>0$ and $\eps>0$ with $0<T_1\leq T_0$ and $0<\eps<1$ such that
\[ \left|\frac{\partial}{\partial_y}q(t,s,y,\xi)-I\right|\leq 1-\eps, \quad \text{ for } s,t\in [0,T_1],\; y, \eta\in \R^d, \; |y|+|\xi|\geq M.\]
As a consequence, the invertibility of $q(t,s,\cdot,\xi)$ follows. The desired estimates can be derived after noticing that 
\[ y(t,s,x,\xi)-x=y(t,s,x,\xi)-q(t,s,y(t,s,x,\xi),\xi),\]
following the same approach of \cite{Yagdjian}, employing the estimates in Lemma \ref{lem:estimates_p,q}. The same idea can be used to prove the existence of the inverse function $\eta=\eta(t,s,x,\xi)$ and the corresponding estimates.
\end{proof}
Now, we deal with the construction of the phase function $\phi = \phi(t,s,x,\xi)$ solving the eikonal equation
\begin{equation} 
\label{eq:eikonal_equation}
\begin{cases} 
&\pa_t \phi(t,s,x,\xi)-\vartheta(t,x,\nabla_x \phi(t,s,x,\xi))=0,\\
&\phi(s,s,x,\xi)= x \cdot \xi\,.
\end{cases}
\end{equation}

\begin{lem}\label{phase_inhomogeneous}
Let $\phi = \phi(t,s,x,\xi)$  be defined as 
\begin{equation}
\label{eq:phase_representation}
\phi(t,s,x,\xi) = v(t,s, y(t,s,x,\xi),\xi) 
\end{equation}
where 
\beqst v(t,s,y,\eta)= y\cdot \eta - \il^t_s \Big(p \cdot \nabla_\xi \vartheta - \vartheta\Big) \Big(\tau , q(\tau ,s,y,\eta),p(\tau,s,y,\eta)\Big)d\tau\;. \eeqst
Then, $\phi$ solves the Cauchy problem \eqref{eq:eikonal_equation}. Moreover, for all $\alpha,\beta \in \N^d$, there exist positive constants $C_{\alpha\beta}$ such that, for all $(x,\xi)\in \R^{2d}$, it holds
\begin{equation}
\label{eq:phase_estimate_pseudo}
\Big|D_x^\alpha D_\xi^\beta \Big(\phi(t,s,x,\xi)-x\cdot \xi\Big)\Big|\leq C_{\alpha\beta}\<x\>^{\eps-|\alpha|}\<\xi\>^{\eps-|\beta|},
\end{equation}
if  $0\leq s,t\leq t_{x,\xi}$, with $\eps>0$ arbitrarily small, and  
\begin{equation}
\label{eq:phase_estimate_hyp}
\Big|D_x^\alpha D_\xi^\beta \Big(\phi(t,s,x,\xi)-x \cdot \xi\Big) \Big| \le C_{\al\beta} \norm{x}^{1-|\alpha|}\norm{\xi}^{1-|\beta|} |\Lambda(t)-\Lambda(s)|,
\end{equation} 
if $\max (s,t) \ge \tilde{t}_{x,\xi}$.
\end{lem}
\begin{proof}
The fact that the function $\phi$ defined in \eqref{eq:phase_representation} solves the Cauchy problem \eqref{eq:eikonal_equation} follows by classical results (see, for instance, \cite{Kumano-go}). The desired estimates follow by \eqref{eq:phase_representation}, employing Faa' di Bruno formula to evaluate the derivatives of composite functions, together with the estimates obtained in Lemma \ref{lem:estimates_p,q}. In particular, in estimate \eqref{eq:phase_estimate_pseudo} we use that 
\[ \<x\>^{\frac{1}{2}}\<\xi\>^\frac{1}{2}\max\{\sqrt{\Lambda(t)},\sqrt{\Lambda(s)}\}\leq \sqrt{N\ln(\<x\>\<\xi\>)},\]
being $(s,x,\xi)$ and $(t,x,\xi)$ in $Z_\pd(N)$.
\end{proof}
From now on, we denote by $\phi^-$ and $\phi^+$ the solutions to \eqref{eq:eikonal_equation}, with $\vartheta$ equal to the real part of $\mathfrak{t}_1$ and,  respectively, the real part of $\mathfrak{t}_2$.

\begin{lem}\label{sym_log}
Suppose $\psi_{s,t}(x,\xi)=(y,\eta)=(\psi^1_{s,t}(x,\xi),\psi^2_{s,t}(x,\xi))$ be the canonical relation associated with the Hamiltonian $\vartheta=\vartheta(t,x,\xi)$ and let $$b(t,x,\xi)=e^{i\int^{T}_{t}a(\tau,x,\xi)d\tau},$$
with $a(t,x,\xi)\in \SGH{0}{0}{1}{0}{N}$ supported in $Z_{\hyp}(N)$.
Then,
$$
c(t,x,\xi)=b(t, \psi_{s,t}(x,\xi))\in \SGH{0}{0}{0}{0}{N_1}
$$
for $0\leq s < t\leq T$ and sufficiently large $N_1\geq N$.
\end{lem}
\begin{proof}
We assume $T$ sufficiently small in $Z_{\hyp}(N)$ to ensure that the canonical transformation $\psi_{s,t}$ is a $\SG$-diffeomorphism. This
implies, in particular, $\psi^1_{s,t} \in S^{1,0}$, 
$\psi^2_{s,t} \in S^{0,1}$, $\jap{\psi^1_{s,t}(x,\xi)}\asymp\jap{x}$
and $\jap{\psi^2_{s,t}(x,\xi)}\asymp\jap{\xi}$, uniformly with respect
to $s,t$, $0\le s\le t\le T$. 

As the symbol is a constant in the pseudodifferential zone, we can assume, without loss of generality, that we are working in the hyperbolic zone, i.e., $t\geq t_{x,\xi}$. Of course, since $a$ is uniformly bounded,
$|c(t,x,\xi)|\lesssim 1$. Let us now estimate the derivatives of first
order. By the chain rule and the hypotheses, for $j=1,\dots,d$,
\begin{align*}
\left|\partial_{x_j}c(t,x,\xi)\right| &= \left|c(t,x,\xi)\right|
\!\cdot\!\left|
\sum_{k=1}^d\int_t^T\!\!\!\left[(\partial_{x_k}a)(\tau, \psi_{s,t}(x,\xi))
(\partial_{x_j}\psi^{1k}_{s,t})(x,\xi) +\right.\right.
\\
&\left.\phantom{|c(t,x,\xi)|
\cdot|
\sum_{k=0}^d\int_t^T[} \!\!\!\!
\left.+(\partial_{\xi_k}a)(\tau, \psi_{s,t}(x,\xi))
(\partial_{x_j}\psi^{2k}_{s,t})(x,\xi)
\right]d\tau
\right| 
\\
&\lesssim
\sum_{k=1}^d\int_t^T\left[
\lambda(\tau)
\jap{\psi^1_{s,t}(x,\xi)}^{-1}
\jap{\psi^2_{s,t}(x,\xi)}^{0}
\jap{x}^0\jap{\xi}^0+
\right.
\\
&\phantom{\lesssim
\sum_{k=1}^d\int^T}
\left.
\!\!+\lambda(\tau)
\jap{\psi^1_{s,t}(x,\xi)}^{0}
\jap{\psi^2_{s,t}(x,\xi)}^{-1}
\jap{x}^{-1}\jap{\xi}
\right]d\tau
\\
&\lesssim\left(\jap{x}^{-1}+ \jap{\xi}^{-1} \jap{x}^{-1} \jap{\xi}\right)
\int_t^T\lambda(\tau)d\tau
\\
&\le \left[\Lambda(T)-\Lambda(t_{x,\xi})\right]\jap{x}^{-1}
\lesssim\jap{x}^{-1}.
\end{align*}
In a completely similar fashion, we obtain, for $j=1,\dots,d$,
the estimates 
$$
\left|\partial_{\xi_j}c(t,x,\xi)\right|\lesssim \norm{\xi}^{-1}.
$$
Finally, by the estimates for the components of the Hamiltonian flow, 
proved above, with $N_1\ge N$ sufficiently large, we find
\begin{align*}
	|\partial_t c(t,x,\xi)|&=\left|c(t,x,\xi)\right|\!\cdot\!
	\left| -a(t,\psi_{s,t}(x,\xi))+
\phantom{
\sum_{k=0}^d\int_t^T[}\right.
\\
& \left.\phantom{|c(t,x,\xi)|
\cdot|\;\;}
+
\sum_{k=1}^d\int_t^T\!\!\!\left[(\partial_{x_k}a)(\tau, \psi_{s,t}(x,\xi))
(\partial_{t}\psi^{1k}_{s,t})(x,\xi) +\right.\right.
\\
&\left.\phantom{|c(t,x,\xi)|
\cdot|
\sum_{k=0}^d\int_t^T[} \!\!\!\!
\left.+(\partial_{\xi_k}a)(\tau, \psi_{s,t}(x,\xi))
(\partial_{t}\psi^{2k}_{s,t})(x,\xi)
\right]d\tau\right|
\\
&\lesssim\lambda(t)+
\\
&+
\sum_{k=1}^d\int_t^T\left[
\lambda(\tau)
\jap{\psi^1_{s,t}(x,\xi)}^{-1}
\jap{\psi^2_{s,t}(x,\xi)}^{0}
\lambda(t)\jap{x}\jap{\xi}^0+
\right.
\\
&\phantom{\lesssim
\sum_{k=1}^d\int^T}
\left.
\!\!+\lambda(\tau)
\jap{\psi^1_{s,t}(x,\xi)}^{0}
\jap{\psi^2_{s,t}(x,\xi)}^{-1}
\lambda(t)\jap{x}^{0}\jap{\xi}
\right]d\tau
\\
&\lesssim\lambda(t)\left[1+\left(\jap{x}^{-1}\jap{x}+ \jap{\xi}^{-1} \jap{\xi}\right)
\int_t^T\lambda(\tau)d\tau \right]
\\
&\lesssim\lambda(t)
\lesssim \frac{\lambda(t)}{\Lambda(t)}\ln\frac{1}{\Lambda(t)}.
\end{align*}
Notice that the last inequality is equivalent to $\displaystyle 1\lesssim 
-\frac{\ln\Lambda(t)}{\Lambda(t)}$, which holds true for $t\ge t_{x,\xi}$.
The estimates for the derivatives of higher arbitrary order follow inductively, by means of the F\'aa di Bruno formula.
\end{proof}

\subsection{Construction of the amplitudes}\label{determination_E2}
As usual, we look for operator families $E^\mp_2(t,s)$ of the form
\beqst 
E_2^\mp (t,s)w(x) = \il_{\kR^d}\il_{\kR^d} e^{i(\phi^\mp(t,s,x,\xi)-y.\xi)}e_2^\mp (t,s,x,\xi)w(y)dyd\xi , w \in \mathscr{S}(\R^d),
\eeqst
with \beqst \phi^\mp(s,s,x,\xi)= x \cdot \xi\;,\;e_2^\mp(s,s,x,\xi)=1. \eeqst 
The asymptotic representation 
\beqst &&e_{2}^\mp(t,s,x,\xi) \sim \sum^\iy_{j=0} e^\mp_{2,j}(t,s,x,\xi) \quad {\rm modulo}\;\; C\Big([0,T_0]^2,S^{-\iy,-\iy}\Big),\hspace*{2cm}\\ &&e_{2,0}^\mp(s,s,x,\xi) = 1\;,\; e^\mp_{2,j}(s,s,x,\xi)=0 \quad {\rm for}\;\; j \ge 1, 
\eeqst 
allows us to derive the transport equation. Namely, we need to study the action of $D_t - \Op(\mathfrak{t}_1(t))$ and $D_t - \Op(\mathfrak{t}_2(t))$, respectively, on $E_2^-\;,\; E_2^+$. We confine ourselves to the fundamental solution $E_{2}^{-}(t,s)$, corresponding to the amplitude function $e_{2}^{-}(t,s,x,\xi)$: the computations for the case $E^+_2(t,s)$ are similar. Formally, we can of course write
\begin{small}
\beqst
D_tE^{-}_2(t,s)w(x)=\il_{\kR^d}\il_{\kR^d}e^{i(
\phi^{-}(t,s,x,\xi)-y\cdot\xi)}\textstyle{\Big(\pa_t \phi^-
\sum\limits^\iy_{j=0} e^{-}_{2,j}} + \textstyle{\frac{1}{i}} \pa_t
\sum\limits^\iy_{j=0} e^{-}_{2,j} \Big)w
(y)dyd\xi.
\eeqst 
\end{small}
In order to find the expansion terms $e_{2,j}^{-}$, $j=0,1,2,\dots$,
we introduce the abbreviations 
\begin{eqnarray}
\nonumber g_\alpha(t,s,x,\xi)&=&\partial_\xi^\alpha\mathfrak{t}_{1}(t,x,\nabla_{x}\phi^{-}(t,s,x,\xi)), \quad \text{for }|\alpha|\geq 1,\\
\nonumber g_0(t,s,x,\xi)&=&-i\sum_{|\alpha|=2}\frac{1}{\alpha!}g_\alpha(t,s,x,\xi)\partial^{\alpha}_{x}\phi^{-}(t,s,x,\xi),\\
\nonumber Z(t,s)&=&D_{t}-\sum_{|\alpha|=1} g_\alpha(t,s,x,\xi)\partial_{x}^\alpha+g_0(t,s,x,\xi).
\end{eqnarray}
Then, according to the asymptotic expansion for the compositions, 
we obtain the equations
$$
Z(t,s)e^{-}_{2,0} = 0 \textnormal{  and  } Z(t,s)e^{-}_{2,j}+r_{j-1} = 0,
$$
for $0\leq t<T$, with the initial conditions 
$$
e^{-}_{2,0}(s,s)=1 \textnormal{ and } e^{-}_{2,j}(s,s)=0,\ j=1,2,\dots,
$$
where 
\begin{align*}
r_j(t,s,x,\xi)&=\sum_{|\alpha|>2} g_\alpha(t,s,x,\xi)\sum_{\substack{|\alpha_1|+|\alpha_2|=|\alpha| \\ |\alpha_1|\neq 1}}D^{\alpha_1}_x \phi^-(t,x,y)D_x^{\alpha_2} e_{2,j}^-(t,s,x,\xi)\\ &+\sum_{|\alpha|=2} g_\alpha(t,s,x,\xi)D_x^\alpha e_{2,j}^-(t,s,x,\xi).
\end{align*}
The solutions $e^{-}_{2,j}(t,s,x,\xi)$, $j=0,1,2,\dots$, to the above initial value problems are
\begin{eqnarray}
\nonumber e^{-}_{2,0}(t,s,x,\xi)&=&\textnormal{exp}[-i\int_{s}^{t}g_0(\sigma, s, q(\sigma,s,y(t,s,x,\xi)),\xi)d\sigma], \\
\nonumber e^{-}_{2,j}(t,s,x,\xi)&=&-i\int_{s}^{t}r_{j-1}(\sigma, s, q(\sigma,s,y(t,s,x,\xi)),\xi)\\
\nonumber                  &  &\quad\quad\times \textnormal{exp}[-i\int_{s}^{t}g_0(\sigma', s, q(\sigma',s,y(t,s,x,\xi)),\xi)d\sigma']d\sigma.
\end{eqnarray}
\begin{lem}\label{lem:amplitude_E2}
The following useful estimates hold:
\begin{itemize}
\item If $0\leq s, t \leq t_{x,\xi}$ then 
\begin{equation*}
|D_x^\alpha D_\xi^\beta g_\alpha(t,s,x,\xi)|\lesssim \<x\>^{\frac{1}{2}+\eps-|\alpha|}\<\xi\>^{-\frac{1}{2}+\eps-|\beta|}, \quad \text{for } |\alpha|=1,
\end{equation*}
and
\begin{align*}
|D_x^\alpha D_\xi^\beta g_0(t,s,x,\xi)|&\lesssim \<x\>^{-1+\eps-|\alpha|}\<\xi\>^{-1+\eps-|\beta|}|t-s|,
\\
|D_x^\alpha D_\xi^\beta e_0^\mp(t,s,x,\xi)|&\lesssim \<x\>^{-1+\eps-|\alpha|}\<\xi\>^{-1+\eps-|\beta|}|t-s|^2,
\\
|D_x^\alpha D_\xi^\beta r_0(t,s,x,\xi)|&\lesssim \<x\>^{-2+\eps-|\alpha|}\<\xi\>^{-2+\eps-|\beta|}|t-s|^2;
\end{align*}
\item if $\max\{s,t\}\geq \tilde{t}_{x,\xi}$ then
\begin{equation*}
|D_x^\alpha D_\xi^\beta g_\alpha(t,s,x,\xi)|\lesssim \lambda(t)\<x\>^{1-|\alpha|}\<\xi\>^{-|\beta|}, \quad \text{for } |\alpha|=1,
\end{equation*}
and
\begin{align*}
|D_x^\alpha D_\xi^\beta g_0(t,s,x,\xi)|&\lesssim \<x\>^{-|\alpha|}\<\xi\>^{-|\beta|}\lambda(t)|\Lambda(t)-\Lambda(s)|,\\
|D_x^\alpha D_\xi^\beta e_0^\mp(t,s,x,\xi)|&\lesssim \<x\>^{-|\alpha|}\<\xi\>^{-|\beta|}|\Lambda(t)-\Lambda(s)|^2,\\
|D_x^\alpha D_\xi^\beta r_0(t,s,x,\xi)|&\lesssim \<x\>^{-1-|\alpha|}\<\xi\>^{-1-|\beta|}\lambda(t)|\Lambda(t)-\Lambda(s)|^2.
\end{align*}
\end{itemize}
\end{lem}
As a consequence, by using induction on $j\geq 0$ we get the following statement.
\begin{prop}\label{prop_E2}
The parametrix $E_2(t,s)= \; {\rm diag} \; \Big(E^-_2(t,s), E_2^+(t,s)\Big)$ is a diagonal matrix of Fourier integral operators with 
\beqst 
E_2^\mp (t,s)w(x) = \il_{\kR^d}\il_{\kR^d} e^{i(\phi^\mp(t,s,x,\xi)-y\cdot\xi)}e_2^\mp (t,s,x,\xi)w(y)dyd\xi, \\
\phi^\mp(s,s,x,\xi) = x \cdot \xi\; ,\; e_2^\mp(s,s,x,\xi)=1\;. \hspace*{1cm}
\eeqst
The phase functions $\phi^\mp$ satisfy estimate \eqref{eq:phase_estimate_pseudo} and \eqref{eq:phase_estimate_hyp}.
Further, for $(t,x,\xi)\in Z_{\pd}(N),$ the amplitude functions $e^{\mp}_{2}$ satisfy
\begin{equation*}
|D_x^\alpha D_\xi^\beta e^\mp_{2,j}(t,s,x,\xi)|\lesssim \<x\>^{-1-j+\eps-|\alpha|}\<\xi\>^{-1-j+\eps-|\beta|}|t-s|.
\end{equation*}
Moreover, there exists a positive number $N_{0}$ such that, for $(t,x,\xi)\in Z_{\hyp}(N_0)$, it holds
\begin{equation*}
|D_x^\alpha D_\xi^\beta e_{2,j}^\mp(t,s,x,\xi)|\lesssim \<x\>^{-|\alpha|}\<\xi\>^{-|\beta|}\Lambda(t)^j|\Lambda(t)-\Lambda(s)|
\end{equation*}
for $j=0,1,2,\dots$
\end{prop}
\begin{rem}
We may write 
\[ E_2^{\mp}(t,s)=\Op(a^\mp(t,s))+\Op_{\varphi^\mp(t,s)}(b^\mp(t,s)),\]
where $\Op(a^\mp(t,s))$ is a pseudodifferential operator with amplitude
\[ a^\mp(t,s,x,\xi):=\chi\left(\frac{\Lambda(t)\norm{x}\norm{\xi}}{2N\ln (\norm{x}\norm{\xi})}\right)e^{i(\varphi^\mp(t,s,x,\xi)-x\cdot\xi)}e_2^\mp(t,s,x,\xi),\]
and $\Op_{\varphi^\mp(t,s)}(b^\mp(t,s))$ is the $\SG$ Fourier integral operator of type I (see Definition \ref{def:sgfios}) with phase $\phi^\mp$ and amplitude 
$$ b^\mp(t,s,x,\xi)=\left(1-\chi\left(\frac{\Lambda(t)\norm{x}\norm{\xi}}{2N\ln (\norm{x}\norm{\xi})}\right)\right)e_2^\mp(t,s,x,\xi).$$
We note that, as a consequence of \eqref{eq:phase_estimate_pseudo} and Lemma \ref{lem:amplitude_E2}, the symbol $a^\mp$ belongs to $L_\infty([0,T_0]^2, S^{-1+\eps,-1+\eps}_{(\eps)})$, according to the definition given in \eqref{eq:generalSGclass}.
\end{rem}

\begin{lem}\label{sym_estimate_reg}
Let $a(t,x,\xi)\in \SGO{-p}{-p}{-p}{p+k}{N}$ be supported in the regular zone defined in \eqref{regular_zone} with $t_{x,\xi}'$. Then, the symbol $$b(t,x,\xi)=e^{i\int^{T}_{t}a(\tau, x, \xi)d\tau}$$ satisfies 
$$
|\partial^{\alpha}_{x}\partial^{\beta}_{\xi}b(t,x,\xi)|\leq C_{\alpha\beta}\norm{x}^{-|\alpha|}\norm{\xi}^{-|\beta|}
$$
for any $p\geq k$. 
\end{lem}
\begin{proof}
We need to prove the estimate on in $Z_{\reg}(N)$ as $b(t,x,\xi)$ is constant outside $Z_{\reg}(N)$. For $|\alpha|=|\beta|=0$, the result clearly holds. We can then proceed inductively, 
by observing that, for some positive constants $C_{j}$, $C'_{j}$, we have 
$$
\Big|\partial_{x_j}b(\tau,x,\xi)\Big| \leq C_{j}\int^{T}_{t}|\partial_{x_j}a(\tau,x,\xi)|d\tau.
$$
By the estimate on the symbols in $\SGO{-p}{-p}{-p}{p+k}{N}$ we have
\begin{align}
\nonumber  \Big|\partial_{x_j}b(\tau,x,\xi)\Big| &\leq  C_{j}\norm{x}^{-1}\int ^T_{t} \textstyle{\frac{1}{\norm{x}^{p}\norm{\xi}^{p}\lambda(\tau)^{p}}\Big(\frac{\lambda(\tau)}{\Lambda(\tau)} \ln \frac{1}{\Lambda(\tau)}\Big)^{p+k} d\tau }\\
\nonumber &\leq C_{j}\norm{x}^{-1}\frac{ (\ln \norm{x}\norm{\xi})^{(p+k)}}{(\norm{x}\norm{\xi} \Lambda(t_{x,\xi}'))^p}. 
\end{align}
In $Z_{\reg}(N)$ we can simplify the above expression to  
$$
\Big|\partial_{x_j}b(\tau,x,\xi)\Big| \leq \norm{x}^{-1}\frac{C_{j}}{(2N)^p} (\ln \norm{x}\norm{\xi})^{(p+k)-2p} \leq C'_{j}\norm{x}^{-1}  \textnormal{   as   }p\geq k.
$$
Similarly, we have 
$
|\partial_{\xi_j}b(\tau,x,\xi)|\leq C'_{j}\norm{\xi}^{-1}.
$
The estimates for the higher order derivatives are proved inductively, 
by means of the hypotheses and the F\'aa di Bruno formula.
\end{proof}

\begin{thm}[Egorov's Theorem]\label{egorov}
Let $E_{2}^\mp(t,s)$ be the fundamental solutions given in Proposition \ref{prop_E2} with $0\leq s\leq t$. Assume that $p\in \SGH{0}{0}{0}{0}{N}$ and it is supported in $Z_{\hyp}(N)$. Then, for a sufficiently large $N_{1}$, the operator
$$ 
\Op(p_{1}(t,s))=E_{2}^\mp(s,t)\Op(p(t))E_{2}^\mp(t,s)
$$ 
is a pseudodifferential operator with symbol $p_{1}(\cdot,s,\cdot,\cdot)\in \SGH{0}{0}{0}{0}{N_1}$ defined for all $s\in[0,t)$ for $t\in [0,T_1]$ and vanishing in the pseudodifferential zone. 
\end{thm}
The proof follows in similar lines to that of the standard Egorov's theorem, cf. \cite{Yagdjian}. For the sake of completeness, we illustrate the argument.
%
\begin{proof}
Let $s$ be fixed while $t$ runs over $[0,s]$, then we denote 
$$
P_1(t):=P_1(s,t)=E_{2}^-(t,s)\Op(p(s))E_{2}^-(s,t),\quad 0\leq t\leq s.
$$
Clearly, $P_1(t)$ solves, modulo smoothing operators, the initial value problem 
\begin{equation}\label{eqn_Pst}
D_{t}P_1(t)=\Op(\mathfrak{t}_1(t))P_1(t)-P_1(t)\Op(\mathfrak{t}_1(t)), P_1(s)=\Op(p(s)).
\end{equation}
We are going to construct an approximating solution, $Q(t)$ to (\ref{eqn_Pst}) in the form of a pseudodifferential operator $Q(t):=Q(s,t) =\Op(q(t,s))$ with a parameter dependent symbol $q(t,s)$ belonging to $S^{0,0}$. So we determine $Q(t)$ solving the Cauchy problem 
$$
\frac{\partial}{\partial t}Q(t)=i\Op(\mathfrak{t}_1(t))Q(t)-iQ(t)\Op(\mathfrak{t}_1(t))+R(t),\quad Q(s)=\Op(p(s)),
$$
where $R(t)$ is a smooth family of regularizing operators on $\mathbb{R}^{d}$. 

In order to solve for $q(t, s, x, \xi)$, we assume an asymptotic expansion in the form 
$$
q(t,s,x,\xi)\sim q_{0}(t,s,x,\xi)+q_{1}(t,s,x,\xi)+\cdots,
$$
where $q_{j}(t,s,x,\xi)\in C^{\infty}([0,T]; S^{-j,-j}(\mathbb{R}^{d}\times \mathbb{R}^{d}))$. Therefore, we first solve the transport equation 
$$
\nonumber \frac{\partial q_{0}}{\partial t}-\sum^{d}_{j=1}\frac{\partial \phi}{\partial \xi_{j}}\frac{\partial q_{0}}{\partial x_{j}}+\sum_{j=0}^{d}\frac{\partial \phi}{\partial x_{j}}\frac{\partial q_{0}}{\partial \xi_{j}}=0,
$$
or, written equivalently by means of the Hamiltonian flow, 
$$\left(\frac{\partial }{\partial t}-\mathcal{H}_{\mathfrak{t_1}}\right)q_0=0,$$ 
with the initial condition $q_{0}(s,s,x,\xi)=p(s,x,\xi)$.

As $p(t,s,x,\xi)$ vanishes on the pseudodifferential zone, evidently, $q_{0}(t,s,x,\xi)=0 \textnormal{ for all } (s,x,\xi)\in Z_{\pd}(N), t\in [0,s]$. To solve for $q_{0}(t,s,x,\xi)$ in the hyperbolic zone, we observe that, by choosing $k$ and $T_{1}$ sufficiently small, we have a $\SG$-diffeomorphism, that is, $(y,\eta)\rightarrow (x,\xi)$ with  
$$
(x,\xi)=(x(t,s,y,\eta), \xi(t,s,y,\eta)): \mathbb{R}^{d}_{y}\times \mathbb{R}^{d}_{\eta} \rightarrow \mathbb{R}^{d}_{x}\times \mathbb{R}^{d}_{\xi}.
$$
For the parameters $t,s$ satisfying $t\leq s< T_{1}$ there exists an inverse mapping $(x,\xi)\rightarrow (y,\eta)$ with 
$$
(y,\eta)=(y(t,s,x,\xi), \eta(t,s,x,\xi)): \mathbb{R}^{d}_{x}\times \mathbb{R}^{d}_{\xi} \rightarrow \mathbb{R}^{d}_{y}\times \mathbb{R}^{d}_{\eta},
$$ 
for all $(t,x,\xi)\in Z_{\pd}(N), (s,x,\xi)\in Z_{\hyp}(N), t,s\leq T_{1}$. Moreover, the function $q_{0}(t,s,x,\xi)$ is constant on the integral curves of the Hamiltonian flow corresponding to 
$\mathfrak{t}_1$.  Thus, by the initial condition $q_{0}(s,s,x,\xi)=p(s,x,\xi)$ we have
$$
q_{0}(t,s,x,\xi)=p(s,y(t,s,x,\xi), \eta(t,s,x,\xi)).
$$ 
Similarly, we can determine the lower order terms $q_{l}(t,s,x,\xi)$ for $l=1,2,\dots$ by the equation
$$
\frac{\partial q_{l}}{\partial t}-\sum^{d}_{j=1}\frac{\partial \phi}{\partial \xi_{j}}\frac{\partial q_{l}}{\partial x_{j}}+\sum_{j=0}^{d}\frac{\partial \phi}{\partial x_{j}}\frac{\partial q_{l}}{\partial \xi_{j}}= \zeta_{l}(t,s,x,\xi)
$$ 
with  
$$
\zeta_{l}(t,s,x,\xi)=\sum_{i=0}^{l-1}\sum_{|\alpha|=l-i+1}\frac{1}{\alpha !}\left(D^{\alpha}_{\xi}\phi\partial^{\alpha}_{x}q_{i}-D^{\alpha}_{\xi}q_{i}\partial^{\alpha}_{x} \phi\right),\quad l=1,2,\dots
$$
and the initial condition $q_{k}(s,s,x,\xi)=0$.

By a variant of the usual asymptotic expansion argument, we can show that there exists a symbol $q(t,s,x,\xi)\sim \sum^{\infty}_{l=0}q_{l}(t,s,x,\xi)$. Furthermore, by Lemma \ref{hamiltonian_flow1} and Lemma \ref{sym_log}, we have $$q(t,s,x,\xi)\in \SGH{0}{0}{0}{N_1}.$$ The uniqueness of the solution modulo a smooth family of regularizing 
operators $R(t)$ is also a consequence of our approach. The proof is complete.  
\end{proof}

\subsection{Parametrix for the diagonal terms}
In this section we will construct the parametrix to 
\begin{equation}\label{parametrix_three_term}
D_{t}E_{1}-\mathcal{D}E_{1}+\mathcal{D}_{2}E_{1}= 0,\quad E_{1}(s,s)= I.
\end{equation}
The equalities of course holds modulo smoothing terms. Using $E_2=E_2(t,s)$ from the previous section, we define
 \beqst \label{4.8} E_1(t,s) = E_2(t,s) Q_1(t,s),\;\; Q_1(s,s) = I\,. \eeqst 
 Substituting into (\ref{parametrix_three_term}) gives the
Cauchy problem \beq \label{4.9} D_tQ_1 + E_2(s,t)\mathcal{D}_{2}(t) E_2(t,s) Q_1 = 0\;,\;\; Q_1(s,s) = I\;. \eeq 

By the Egorov's Theorem \ref{egorov}, we have that the matrix-valued operator $R_1(t,s) =E_2(s,t)\mathcal{D}_{2}(t)E_2(t,s)$ consists of pseudodifferential operators
with principal symbol 
$$r_1= r_1(t,s,x,\xi) = \sigma_p(\mathcal{D}_{2})\Big(t,
\psi_{s,t}(x,\xi)\Big).$$

Since $\sigma(\mathcal{D}_2)$ belongs to  $\SGH{0}{0}{0}{0}{N} + \SGO{-1}{-1}{-1}{2}{N}$ and vanishes in $Z_{\pd}(N)\cup Z_{\osc}(N)$, we may write 
\[ \mathcal{D}_{2}= \mathcal{D}_{2,0} + \mathcal{D}_{2,1}, \] 
where
\[d_{2,0}:=\sigma(\mathcal{D}_{2,0}) \in \SGH{0}{0}{0}{0}{N}, \quad d_{2,0}\equiv 0 \text{ in } Z_{\pd}(N); \]
and 
\[d_{2,1}:=\sigma(\mathcal{D}_{2,1}) \in \SGO{-1}{-1}{-1}{2}{N}, \quad  \sigma(\mathcal{D}_{2,1}) \equiv 0 \text{ in } Z_{\pd}(N)\cup Z_\osc(N).\] 
As a consequence of Lemma \ref{lem:hamiltonian_inverse}, we see that the compositions $d_{2,0}(t,\psi_{s,t})$ and $d_{2,0}(t,\psi_{s,t})$ satisfy
\begin{align*}
|D_x^\alpha D_\xi^\beta d_{2,0}(t,\psi_{s,t}(x,\xi))|&\lesssim \Lambda(t) \<x\>^{-|\alpha|}\<\xi\>^{-|\beta|},\\
|D_x^\alpha D_\xi^\beta d_{2,1}(t,\psi_{s,t}(x,\xi))|&\lesssim  \<x\>^{-1-|\alpha|}\<\xi\>^{-1-|\beta|}\frac{\lambda(t)}{\Lambda(t)}\bigg(\ln\bigg(\frac{1}{\Lambda(t)}\bigg)\bigg)^2.
\end{align*}
In particular, since  $\Lambda(t)\leq T\lambda(t)$ and, for all $(t,x,\xi)\in Z_\hyp(N)$, it holds
\[ \<x\>^{-1}\<\xi\>^{-1}\frac{\lambda(t)}{\Lambda(t)}\bigg(\ln\bigg(\frac{1}{\Lambda(t)}\bigg)\bigg)^2\lesssim \lambda(t),\]
we may estimate 
\begin{equation}
\label{eq:amplitude_second_step}
|D_x^\alpha D_\xi^\beta r_1(t,s,x,\xi)|\lesssim \lambda(t) \<x\>^{-|\alpha|}\<\xi\>^{-|\beta|}
\end{equation}
for all $(t,x,\xi)\in Z_\hyp(N)$. The obtained estimates allows to prove that the fundamental solution to \eqref{4.9} is given by a parameter-dependent
pseudodifferential operator
\beq \label{4.10} 
\quad Q_1(t,s)w(x)= \il_{\kR^n} e^{i\: x \cdot
\xi}q_1(t,s,x,\xi) \hat{w} (\xi)d\xi,\;\; q_1(s,s,x,\xi)=1. 
\eeq
Let us determine the matrix amplitude $q_1$ by means of an asymptotic expansion  \beqst q_1(t,s,x,\xi) \sim \sum^\iy_{j=0}
q_{1,j}(t,s,x,\xi). \eeqst 
Inserting \eqref{4.10} into \eqref{4.9} we get the following system of ordinary differential equations for $q_{1,j}$: 
\beqst  D_t q_{1,0} + r_1(t,s,x,\xi) q_{1,0} &=&0,\quad q_{1,0}(s,s,x,\xi)=1,
\eeqst 
and, for all $j>0$,
\beqst 
D_t q_{1,j} + r_1(t,s,x,\xi) q_{1,j} &=& - \sum^j_{|\al|=1} \textstyle{\frac{1}{\al !}}D^\al_\xi r_1(t,s,x,\xi)\pa^\al_x q_{1,j-|\al|}(t,s,x,\xi), 
\eeqst 
with the initial condition $q_{1,j}(s,s,x,\xi) =0$.
The solution of this system is given by
\beqst q_{1,0}(t,s,x,\xi)&=& \exp\Big(-i \il^t_s
r_1(\tau,s,x,\xi)d\tau\Big),\\ q_{1,j}(t,s,x,\xi)&=& - i \il^t_s
\exp\Big(-i \il^t_\si r_1 (\tau,s,x,\xi)d\tau\Big)\\ && \times
\sum^j_{|\al|=1} \textstyle{\frac{1}{\al !}} D^\al_\xi
r_1(\si,s,x,\xi)\pa^\al_x q_{1,j-|\al|}(\si,s,x,\xi) d\si.
\eeqst 
Employing the estimate \eqref{eq:amplitude_second_step}, we can show that the function $q_{1,j}=q_{1,j}(t,s,x,\xi)$ satisfies
\begin{equation*}
	\label{4.11} \Big|\pa_x^\beta \pa_\xi^\al q_{1,j}(t,s,x,\xi)\Big|
	\le C_{\al \beta } \norm{x}^{-j-|\beta|}\norm{\xi}^{-j-|\al|}|\Lambda(t)-\Lambda(s)|,
\end{equation*}
for each $j=0,1,\dots$


Collecting all the results of this section, we have proved the next Proposition \ref{prop_E4}.

\begin{prop}\label{prop_E4}
The parametrix $E_1 = E_1(t,s)$ to the operator $D_t - {\mathcal D}_1 + \mathcal{D}_{2}$ can be written as $E_1(t,s)= E_2(t,s)Q_1(t,s)$, where $E_2(t,s)$ is the diagonal 
matrix of Fourier integral operator from Proposition \ref{prop_E2} and $Q_1=Q_1(t,s)$ is a diagonal pseudodifferential operator with symbol belonging to $W^1_\iy\Big([0,T_0]^2, S^{0,0}\Big)$.
\end{prop}

\subsection{Parametrix for the full system}
Finally, we devote ourself to find $E=E(t,s)$ such that  
\beqst D_tE- {\mathcal D}E + \mathcal{D}_{2}E + \mathcal{B}_{\infty}E = 0,\quad E(s,s) =
I,\;\eeqst
modulo smoothing terms. Setting, as above, $E(t,s)=E_1(t,s)Q(t,s)$ we have to study 
\beq \label{4.12} D_tQ+E_1(s,t)\mathcal{B}_{\infty}(t)E_1(t,s)Q = 0,\quad Q_1(s,s) = I.\eeq 

The matrix operator $\mathcal{B}_{\infty}$ does not have diagonal form, so Egorov's theorem cannot be applied in this case. 
However, employing the property $\sigma(\mathcal{B}_\infty)\in \mathcal{HG}_N$ (see \eqref{eq:defHGN}), we will be able to 
prove that the composition $E_1(s,t)\mathcal{B}_{\infty}(t)E_1(t,s)$ is still a parameter-dependent pseudodifferential operator 
for all $0\leq s\leq t \in [0,T]$.

Let us first consider the composition $\mathcal{B}_{\infty}(t)E_1(t,s)$ of the matrix $\mathcal{B}_{\infty}=(\mathcal{B}^{jk}_{\infty})$ with the diagonal matrix $E_{1}(t,s)=\; {\rm diag}\; (E_1^-(t,s),E_1^+(t,s))$. Notice that, as a consequence of Proposition \ref{prop_E4}, there exists $e_1^\mp \in L_\infty \Big([0,T_0]^2, S^{0,0}\Big)$ such that 
$$
E_1(t,s)^{\mp} w(x)= \il_{\kR^n}\il_{\kR^n} e^{i (\phi^{\mp} (t,s,x,\xi)-y\cdot\xi)} e_1^{\mp} (t,s,x,\xi) w(y) dyd\xi.
$$ 
Applying Theorem \ref{thm:composition}, we may conclude that $\mathcal{B}_{\infty}(t)E_1(t,s)$ is a matrix-valued Fourier integral operator 
\beqst  \mathcal{B}_{\infty}(t)E_1(t,s)=\il_{\kR^n}\il_{\kR^n} e^{i(\phi^{\mp}(t,s,x,\xi)-y\cdot\xi)}r_{\mp}(t,s,x,\xi)w(y)dyd\xi \eeqst 
with symbols having entries admitting asymptotic expansions of the form
\begin{align*}
r_{\mp}^{jk}(t,&s,x,\xi)\\
& 
\sim \sum_{\alpha}\frac{i^{|\alpha|}}{\alpha !}(D^{\alpha}_{\xi}\si(\mathcal{B}^{jk}_\iy))(t,s,x,\nabla_{x}\varphi^\mp(t,x,\xi))D^{\alpha}_{y}[e^{i\Phi^\mp(t,x,y,\xi)}e_1^\mp(t, s,y,\xi)]_{y=x},
\end{align*}
where $\Phi^\mp(t,x,y,\xi)=\phi^\mp(t,y,\xi)-\phi^\mp(t,x,\xi)+(x-y)\cdot \nabla_{x}\phi^\mp(t,x,\xi)$.

By using Lemma \ref{lem:hamiltonian_path}, Lemma \ref{sym_estimate_reg},  Proposition \ref{prop_E2},  Theorem \ref{egorov} and Proposition \ref{prop_E4}, 
we can obtain for all $\alpha, \beta\in \N^d$, the estimate
\begin{equation*}
\begin{aligned}
|D_x^\alpha &D_\xi^\beta r_\mp(t,s,x,\xi)|\lesssim 
 C_{\alpha\beta}\<x\>^{-|\alpha|} \<\xi\>^{-|\beta|}\\
& \times \begin{cases}
\lambda(t)\<x\>^{-p}\<\xi\>^{-p} \bigg(\frac{1}{\Lambda(t)}\ln\bigg(\frac{1}{\Lambda(t)}\bigg)\bigg)^{p+1}  & \text{ in } Z_\reg(2N),\\
1+\lambda(t)\<x\>^{-1}\<\xi\>^{-1} \bigg(\frac{1}{\Lambda(t)}\ln\bigg(\frac{1}{\Lambda(t)}\bigg)\bigg)^2 & \text{ in } Z_\osc(2N), \\ 
\<x\>\<\xi\> &  \text{ in } Z_\pd(2N),
\end{cases}
\end{aligned}
\end{equation*}
where $p\in \N$ can assume an arbitrary large value.

Since we want to prove that $\mathcal{B}_{\infty}(t)E_1(t,s)$ is a pseudodifferential operator, we write it as 
\[ \mathcal{B}_{\infty}(t)E_1(t,s)=
\il_{\kR^n} \il_{\kR^n} e^{i (x-y) \cdot \xi} e^{i \phi^{\mp}(t,s,x,\xi)-ix \cdot \xi} r_{\mp}(t,s,x,\xi) w(y)dyd\xi, \]
and we prove that 
$$  \tilde{r}_{\mp}(t,s,x,\xi)):= e^{i \phi^{\mp}(t,s,x,\xi)-ix \cdot \xi} r^\mp(t,s,x,\xi)$$
is the symbol of a pseudodifferential operator. \\
We note that the exponential term satisfies:
\begin{equation}
\label{eq:exponential_estimate_Zpd}
 |D_x^\alpha D_\xi^\beta e^{i \phi^{\mp}(t,s,x,\xi)-ix \cdot \xi} |\lesssim C_{\alpha\beta} \<x\>^{-(1-\eps)|\alpha|+\eps|\beta|}\<\xi\>^{\eps|\alpha|-(1-\eps)|\beta|} 
\end{equation}
if $\max\{s,t\}\leq t_{x,\xi}$  and
\[ |D_x^\alpha D_\xi^\beta e^{i \phi^{\mp}(t,s,x,\xi)-ix \cdot \xi} |\lesssim C_{\alpha\beta} \<x\>^{|\beta|}\<\xi\>^{|\alpha|} \Lambda(t)^{|\alpha|+|\beta|},\]
if $t\geq s\geq t_{x,\xi}$.
In particular, if $(t,x,\xi)$ belongs to $Z_\osc(2N)$ we may estimate 
\[ \<x\>^\beta\Lambda(t)^{|\beta|}\leq 2N \<\xi\>^{-|\beta|}\ln(\<x\>\<\xi\>)^{2|\beta|}\lesssim \<x\>^{\eps|\beta|} \<\xi\>^{-(1-\eps)|\beta|},\]
and, similarly, 
\[ \<\xi\>^{|\alpha|} \Lambda(t)^{|\alpha|}\leq 2N \<x\>^{-|\alpha|}\ln(\<x\>\<\xi\>)^{2|\alpha|}\lesssim \<x\>^{-(1-\eps)|\alpha|} \<\xi\>^{\eps|\alpha|},\]
for any $\eps>0$ arbitrarily small. Then, \eqref{eq:exponential_estimate_Zpd} holds also in $Z_\osc(2N)$. Moreover, if $t\geq \max\{s,t'_{x,\xi}\}$, we obtain 
	\begin{align*}
		|D_x^\alpha D_\xi^\beta & \tilde{r}_{\mp}(t,s,x,\xi)| \\
		&\lesssim \sum_{\substack{|\alpha_1|+|\alpha_2|=|\alpha| \\|\beta_1|+|\beta_2|=|\beta|}} \lambda(t)\frac{\<x\>^{|\beta_2|-|\alpha_1|-p}\<\xi\>^{|\alpha_2|-|\beta_1|-p}}{\Lambda(t)^{p+1-|\alpha_2|-|\beta_2|}}\bigg(\ln\bigg(\frac{1}{\Lambda(t)}\bigg)\bigg)^{p+1};
	\end{align*}
	taking $p>1$ sufficiently large we can estimate \[ \Lambda(t)^{p-|\alpha_2|-|\beta_2|}>N\ln(\<x\>\<\xi\>)^{2p-2|\alpha_2|-2|\beta_2|}(\<x\>\<\xi\>)^{-p+|\alpha_2|+|\beta_2|}. \]
	Taking into account that 
	$$ \ln\bigg(\frac{1}{\Lambda(t)}\bigg)\leq \ln(\<x\>\<\xi\>), $$
	we may conclude
	\begin{align*}
		|D_x^\alpha& D_\xi^\beta \tilde{r}_\mp(t,s,x,\xi)|\\
		&\lesssim  \frac{\lambda(t)}{\Lambda(t)}\ln\bigg(\frac{1}{\Lambda(t)}\bigg)\<x\>^{-(1-\eps)|\alpha|+\eps|\beta|}\<\xi\>^{\eps|\alpha|-(1-\eps)|\beta|}\ln(\<x\>\<\xi\>)^{-p} 
	\end{align*}
	for any $\eps>0$ arbitrarily small.\\
	Similarly, since $\sigma(\mathcal{B}_\infty)$ belongs to $\SGH{0}{0}{0}{0}{2N} + \SGH{-1}{-1}{-1}{2}{2N}$ we may estimate for all $(t,x,\xi)\in Z_\osc(2N)$
	\begin{align*}
	|D_x^\alpha D_\xi^\beta \tilde{r}_\mp(t,s,x,\xi)|&\lesssim \<x\>^{-(1-\eps)|\alpha|+\eps|\beta|}\<\xi\>^{\eps|\alpha|-(1-\eps)|\beta|} \\
	&\hspace{40pt}\times\bigg(1+\ln(\<x\>\<\xi\>)^2\<x\>^{-1}\<\xi\>^{-1}\frac{\lambda(t)}{\Lambda(t)^2}\bigg),
	\end{align*}
	for any $\alpha$ and $\beta$ multi-indices. \\
	Finally, since $\sigma(\mathcal{B}_\infty)$ satisfies \eqref{eq:B1_bound_pd} if $(s,x,\xi), (t,x,\xi)\in Z_\pd(2N)$, then we may estimate 
	\begin{align*}
		|D_x^\alpha D_\xi^\beta \tilde{r}_\mp(t,s,x,\xi)|&\lesssim \<x\>^{-(1-\eps)|\alpha|+\eps|\beta|}\<\xi\>^{\eps|\alpha|-(1-\eps)|\beta|} \\
		&\hspace{40pt}\times\bigg(\rho(t,x,\xi)+\frac{\partial_t \rho(t,x,\xi)}{\rho(t,x,\xi)}\bigg),
	\end{align*}
	for any $\alpha$ and $\beta$ multi-indices. 
	Summarizing, we obtain that for all $(s,t)\in [0,T_0]^2$ the operator  $\mathcal{B}_{\infty}(t)E_1(t,s)$ is a pseudodifferential operator with symbol $\tilde{r}_\mp$ satisfying
	\begin{equation*}
		|D_x^\alpha D_\xi^\beta \tilde{r}_\mp(t,s,x,\xi)|\leq
		C_{\alpha\beta}\<x\>^{-(1-\eps)|\alpha|+\eps|\beta|}\<\xi\>^{\eps|\alpha|-(1-\eps)|\beta|}g_p(t,x,\xi)
	\end{equation*}
	where 
	\begin{equation}
		\label{eq:hp}
		\small
		g_pt,x,\xi):=\begin{cases}
			\frac{\lambda(t)}{\Lambda(t)}\ln\Big(\frac{1}{\Lambda(t)}\Big)\ln(\<x\>\<\xi\>)^{-p} & \text{ in } Z_\reg(2N),\\
			1+\ln(\<x\>\<\xi\>)^2\<x\>^{-1}\<\xi\>^{-1}\frac{\lambda(t)}{\Lambda(t)^2} & \text{ in } Z_\osc(2N), \\ 
			\rho(t,x,\xi)+\frac{\partial_t \rho(t,x,\xi)}{\rho(t,x,\xi)} &  \text{ in } Z_\pd(2N);
		\end{cases}
\end{equation} 
In particular, there exist $K_1$, $K_2$, $K_3$, positive constants depending only on $N$, such that 
\begin{align*}
\int_{t'_{x,\xi}}^T g_p(t,x,\xi)\,dt & \leq \ln(\Lambda(t'_{x,\xi}))^2 \ln(\<x\>\<\xi\>)^{-p} \leq K_1, \\
\int_{t_{x,\xi}}^{t'_{x,\xi}} g_p(t,x,\xi)\,dt & \leq \frac{\ln(\<x\>\<\xi\>)^2}{\Lambda(t_{x,\xi})\<x\>\<\xi\>}\leq K_2 \ln(\<x\>\<\xi\>),\\
\end{align*}
and
\begin{align*}
\int_0^{t_{x,\xi}} g_p(t,x,\xi)\,dt & \leq \frac{K_3}{2}\bigg(1+\sqrt{\<x\>\<\xi\>\ln(\<x\>\<\xi\>)}\int_0^{t_{x,\xi}}\frac{\lambda(t)}{\sqrt{\Lambda(t)}}\,dt\bigg) \\
& \leq  K_3 \ln(\<x\>\<\xi\>).
\end{align*}
In the same way we can treat the operator 
\beq \label{4.14}  R(t,s):=E_1(s,t)\Big(\mathcal{B}_{\infty}(t)E_1(t,s)\Big). \eeq  
We have proved the next Proposition \ref{prop:param_pseudo}.

\begin{prop}\label{prop:param_pseudo}
For all $s\leq t \in [0,T_0]$, the matrix operator $R=R(t,s)$ defined by \eqref{4.14} is a pseudodifferential operator whose symbol $\si(R(t,s))$ belongs to $L_\iy\Big([0,T_0]^2, S^{1,1}_{(\ve)}\Big)$ and satisfies 
\begin{equation*}
|D_x^\alpha D_\xi^\beta \si(R)(t,s,x,\xi)|\lesssim 
C_{\alpha\beta}\<x\>^{-(1-\eps)|\alpha|+\eps|\beta|}\<\xi\>^{\eps|\alpha|-(1-\eps)|\beta|}g_p(t,x,\xi),
\end{equation*} 
for every $p \ge 0$ and $\ve >0$ arbitrarily small, with the function $g_p$ defined by \eqref{eq:hp}.
\end{prop}

Following the same approach in \cite[Proposition 3.9.1]{Yagdjian} and in \cite[Proposition 3.8]{Dreher} one can prove the following Lemma \ref{lem:QsigmaR}, that allows to estimate the loss of regularity and the loss of decay due to the \textit{bad behavior} in the pseudodifferential and oscillation zones.
	\begin{lem}\label{lem:QsigmaR}
	Let $Q(t,s)$ the solution, modulo smoothing  operators, to the Cauchy problem 
	\begin{equation} 
			\label{eq:lemYag}
			D_tQ+R(t,s)Q=0, \quad Q(t,s)=I,
	\end{equation}
	where $\sigma(R)\in L_\infty([0,T_0]^2, \SGs^{m,\mu}_{r_1,r_2,\rho_1,\rho_2})$. Assume that, for all $\alpha,\beta \in \N^d$, there exists $C_{\alpha\beta}>0$ such that
	\begin{equation}
		\label{eq:generalSGclass}
		|D^{\alpha}_{x}D_{\xi}^{\beta}\sigma(R)|\leq C_{\alpha\beta}\norm{x}^{-r_1 |\alpha|+r_2 |\beta|}\norm{\xi}^{\rho_1|\alpha|-\rho_2|\beta|}g(t,x,\xi),
	\end{equation}
	for some $g\in C([0,T_0]^2\times \R^{2d})$. Suppose that $g(t,x,\xi)\lesssim \<x\>^\ell\<\xi\>^\omega$ for some $\ell, \omega \in \R$ and it holds
	\[ \int_0^{T_0}g(t,x,\xi)dt\leq K \ln(\<x\>\<\xi\>).\]
	Then, there exists a solution $Q$ to \eqref{eq:lemYag} with matrix symbol $\sigma(Q)$ satisfying
	\begin{equation*}
		|D^{\alpha}_{x}D_{\xi}^{\beta}\sigma(Q)|\leq C_{\alpha\beta}\norm{x}^{K-r_1 |\alpha|+r_2 |\beta|}\norm{\xi}^{K+\rho_1|\alpha|-\rho_2|\beta|}\ln(\<x\>\<\xi\>)^{|\alpha|+|\beta|+1}.
	\end{equation*}
Such solution is unique modulo $C([0,T_0]^2, \Op(S^{-\infty,-\infty}))$.
\end{lem}

As a consequence of the results proved above, 
employing Lemma \ref{lem:QsigmaR}, we finally obtain the concluding result of this section, the next Proposition \ref{prop: loss_K0}.

\begin{prop}\label{prop: loss_K0}
The parametrix $E=E(t,s)$ to the operator $D_t - \mathcal{D} + \mathcal{D}_{2} + \mathcal{B}_{\infty}$ can be written as $E(t,s)= E_1(t,s)Q(t,s)$, where $E_1=E_1(t,s)$ is the 
matrix of Fourier integral operators given in \eqref{4.12} and $Q$ is a matrix of parameter-dependent pseudodifferential operators with symbol belonging to 
$$L_\iy\Big([0,T_0]^2,S^{K_0,K_0}_{\ve}\Big) \cap W^1_\iy\Big([0,T_0]^2,S^{K_0+1+\ve,K_0+1+\ve}_{\ve}\Big),$$ for every small $\ve>0$. Here, the constant $K_0$ 
describes the loss of derivatives and decay coming from the pseudodifferential zone $Z_{\pd}(2N)$ and the oscillations subzone $Z_\osc(2N)$.
\end{prop}
\begin{rem}
We notice that the function $g_p$ in \eqref{eq:hp} is integrable in $[0,T]$. In particular, there exists a constant $K_0$ such that 
\begin{equation}
\label{eq:K0}
 \int_0^T g_p(t,x,\xi) \, dt \lesssim K_0 \ln(\<x\>\<\xi\>).
\end{equation}
This can be shown observing that, since $\Lambda$ is increasing in $[0,T]$, we may estimate
\begin{align*}
	\int_{t'_{x,\xi}}^{T} \frac{\lambda(t)}{\Lambda(t)^{q+1}}\ln\bigg(\frac{1}{\Lambda(t)}\bigg)^{q+1}\,dt &\lesssim \frac{1}{\Lambda(t'_{x,\xi})^{q}}\ln\bigg(\frac{1}{\Lambda(t'_{x,\xi})}\bigg)^{q+1} \\
	&\lesssim \<x\>^{q}\<\xi\>^q \ln(\<x\>\<\xi\>),
\end{align*}
for any exponent $q>0$. The value of $K_0$ in \eqref{eq:K0} determines the loss of derivatives and the loss of decay obtained in Proposition \ref{prop: loss_K0}.
\end{rem}

\section{Solution of the Cauchy problem}\label{solution_cauchy_problem}
\setcounter{equation}{0}
%
%
%

The analysis performed in the previous sections allows now us to prove our main results, namely, well-posedness, regularity and decay of the solutions to \eqref{eq:main}.
The uniqueness of the solution of \eqref{eq:main}
in the weighted Sobolev spaces follows by standard arguments, see \cite{Cord}. The same is true concerning the existence and 
uniqueness of the fundamental solution of the system in \eqref{main_system}, which we have determined modulo smoothing operators. Let us state this result precisely,
in the next Lemma \ref{lem:fundsolsyst}.
\begin{lem}\label{lem:fundsolsyst}
The fundamental solution $F(t,s)$ of \eqref{main_system} has the representation $F(t,s)=E(t,s)+\mathcal{G}_\infty(t,s)$, where $E(t.s)$ comes from
Proposition \ref{diag_2}, Theorem \ref{diag_3}, and Propositions \ref{prop_E4} and \ref{prop: loss_K0}. Moreover, $\mathcal{G}_\infty(t,s)=\Op(g_\infty(t,s))$
with $g\in W^1_\iy\Big([0,T_0]^2,SG^{-\infty,-\infty}\Big)$, for a suitably small $T_0\in(0,T]$.
\end{lem}
\noindent
The next Corollary \ref{cor_5.1} follows then immediately, by the Duhamel's principle.

\begin{cor}\label{cor_5.1}
The unique solution $U=U(t)$ of \eqref{main_system} can be written as \beqst U(t,x)=F(t,0)U_0(x)+\il^t_0 F(t,s)G(s,x)ds. \eeqst
\end{cor}

Corollary \ref{cor_5.1} implies our first main result, Theorem \ref{cor_5.2}, about solutions of the Cauchy problem \eqref{eq:main} in the scale of Sobolev-Kato spaces.

\begin{thm}[Loss of regularity and decay for solutions in weighted Sobolev spaces] \label{cor_5.2}
Consider the Cauchy problem \eqref{eq:main}, where the coefficients $a_j,b_j$, $j=1,\dots,d$, and $c$ satisfy the assumptions in Proposition \ref{prop: equivalence} (A). 
Assume also that the initial data satisfy $\va \in H^{s,\sigma}(\R^d), \psi \in H^{s-1,\sigma-1}(\R^d)$ and that $g \in C\Big([0,T],H^{s,\sigma}(\R^d)\Big)$,
with suitably large $s,\sigma$. Then, for a suitably small $T_0\in(0,T]$, the Cauchy problem \eqref{eq:main} admits a unique solution 
\begin{align*}
u &\in C\Big([0,T_0],H^{s-s_a,\sigma-s_a}(\R^d)\Big) 
\\
&\cap
C^1\Big([0,T_0],H^{s-s_a-1,\sigma-s_a-1}(\R^d)\Big) \cap C^2\Big([0,T_0],H^{s-s_a-2,\sigma-s_a-2}(\R^d)\Big).
\end{align*}
The loss of derivatives and decay $s_a$ 
depends on the constant $K_0$ from Proposition \ref{prop: loss_K0}.
\end{thm}
\begin{proof}
The claim follows immediately by the decomposition $F=E+\mathcal{G}_\infty$ of the fundamental solution of the system \eqref{main_system} from Lemma \ref{lem:fundsolsyst}, 
the relation between its solution $U$ and the solution $u$ of \eqref{eq:main}, the mapping properties of $\SG$ pseudodifferential and Fourier integral operators on the Sobolev-Kato 
spaces, and the properties of $E$, proved in Section \ref{sec:param}.
\end{proof}

Since $\scS(\R^{d})=H^{\infty,\infty}(\mathbb{R}^{d})=\cap_{s,\sigma\in \mathbb{R}}H^{s,\sigma}(\R^d)$, from Theorem \ref{cor_5.2} we deduce our
second and final main result, the next Theorem \ref{cor_5.3}. 

\begin{thm} [$\scS(\R^{d})$-wellposedness]\label{cor_5.3}
Consider the Cauchy problem \eqref{eq:main}, where the coefficients $a_j,b_j$, $j=1,\dots,d$, and $c$ satisfy the assumptions in Proposition \ref{prop: equivalence} (A). 
Assume also that the initial data satisfy $\va,\psi \in\scS(\R^{d})$ and that $g \in C\Big([0,T],\scS(\R^{d})\Big)$. 
Then, for a suitably small $T_0\in(0,T]$, the Cauchy problem \eqref{eq:main} admits a unique solution $u \in C^2\Big([0,T_0],\scS(\R^{d})\Big)$.
\end{thm}

\begin{exmp}
Choose a shape function $\lambda$, satisfying the hypotheses described in Section \ref{subs:lambda}, and
consider the Cauchy problem
\begin{equation}\label{eq:exmplfinal}
	\begin{cases}
		\displaystyle\partial_{tt} u(t,x) + \lambda(t)^2\left[2+\cos\ln\!\left(\frac{1}{\Lambda(t)}\right)\right](1+|x|^2)(1-\Delta_x)u(t,x)=0 \\
		u(0,x)=\phi(x), \quad u_t(0,x)=\psi(x),\rule{0mm}{7mm}
	\end{cases}
\end{equation}
$t\in[0,T]$, $x\in\R^d$. For the operator in \eqref{eq:exmplfinal} we have
\begin{align*}
	a_j(t,x)&=\lambda(t)^2\left[2+\cos\ln\!\left(\frac{1}{\Lambda(t)}\right)\right]\jap{x}^2=c(t,x), 
	\\
	b_j(t,x)&=0, \quad j=1,\dots,d,
\end{align*}
so that
\[
	a(t,x,\xi)=\lambda(t)^2\left[2+\cos\ln\!\left(\frac{1}{\Lambda(t)}\right)\right]\jap{x}^2\jap{\xi}^2.
\]
Evaluating explicitely the two roots $\tau_1$ and $\tau_2$ of the complete symbol $a(t,x,\xi)$ it is easy to derive that Assumption \textbf{(H)} in Proposition \ref{prop: equivalence}
holds true, and then  our theory applies to the operator in \eqref{eq:exmplfinal}. Notice that, with the same approach used in the proof of Proposition \ref{prop: equivalence} it is also possible to prove that the coefficient  $c(t,x)$ satisfies
\[ |D_t^k D_x^\alpha c(t,x)|\leq C_{k\alpha}\lambda(t)^2 \<x\>^{-|\alpha|} \bigg(\frac{\ln\lambda(t)}{\Lambda(t)}\bigg)^2 \bigg(\frac{\lambda(t)}{\Lambda(t)}\ln\bigg(\frac{1}{\Lambda(t)}\bigg)\bigg) ^k,\]
for suitable positive constants $C_{k\alpha}$ depending on $k\in \N$ and $\alpha\in \N^d$. In particular,  in view of the presence of the oscillating $t$-dependent factor, here logarithms indeed appear in the coefficients estimates.

\end{exmp}


%
\end{document}